\documentclass[12pt]{amsart}
\usepackage{geometry}
\usepackage{graphicx}
\usepackage{amscd}
\usepackage{amssymb}
\usepackage{amsmath}
\usepackage[latin1]{inputenc}

\theoremstyle{plain}                                       %
\newtheorem{thm}{\quad Theorem}                            %
\newtheorem{cor}[thm]{\quad Corollary}                     %
\newtheorem{prop}[thm]{\quad Proposition}                  %
\theoremstyle{definition}                                  %
\newtheorem{defi}[thm]{\quad Definition}                   %
\newtheorem{rmk}[thm]{\quad Remark}                        %
\newtheorem{ejem}[thm]{\quad Example}                      %
\newtheorem{alg}[thm]{\quad Algorithm}
\newcommand{\N}{{\Bbb N}}
\newcommand{\K}{{\Bbb K}}

\newcommand{\Z}{{\Bbb Z}}

\newcommand{\al}{{\alpha}}

\newcommand{\tT}{{\widetilde{T}}}
\newcommand{\tit}{{\widetilde{t}}}

\title{Recurrence relations for polynomial sequences via Riordan
matrices}
\author{Ana Luz\'{o}n* and Manuel A. Mor\'{o}n**}

\begin{document}
\maketitle

\address{*Departamento de Matem\'{a}tica Aplicada a los
Recursos Naturales. E.T. Superior de Ingenieros de Montes.
Universidad Polit\'{e}cnica de Madrid. 28040-Madrid,
SPAIN.}

\email{anamaria.luzon@upm.es}

\address{ **Departamento de Geometria y
Topologia. Facultad de Matematicas. Universidad Complutense de
Madrid. 28040- Madrid, SPAIN.}

 \email{mamoron@mat.ucm.es}

\vspace{1cm}

\begin{abstract}
We give recurrence relations for any family of generalized Appell
polynomials unifying so some known recurrences of many classical
sequences of polynomials. Our main tool to get our goal is the
Riordan group. We use the product of Riordan matrices to interpret
some relationships between different families of polynomials.
Moreover using the Hadamard product of series we get a general
recurrence relation for the polynomial sequences associated to the
so called generalized umbral calculus.
\end{abstract}

\vspace{1cm}

 Keywords: Recurrence relation, Riordan matrix, Generalized
Appell polynomials, Polynomials sequences of Riordan type, Umbral
calculus.

 MSC: 11B83, 68W30, 68R05, 05A40



\section{Introduction}

In this paper we obtain recurrence relations for a large class of
polynomials sequences. In fact, we get this for any family of
generalized Appell polynomials \cite{Boas-Buck}. Our main tool to
reach our goal is the so called Riordan group. \cite{Huang},
\cite{Rog78}, \cite{Sha91}, \cite{Spr94}.

This work is a natural consequence of our previous papers
\cite{2ways}, \cite{teo} and \cite{BanPas}, and then it can be
also considered as a consequence of the well-known Banach's Fixed
Point Theorem. We have also to say that some papers related to
this one have recently appeared in the literature \cite{He-H-S}
and \cite{WangWang} but our approach is different from that in
those papers because, our main result herein is the discovering of
a general recurrence relation for sequences of polynomials
associated, naturally, to Rirodan matrices. In particular we get a
characterization of Riordan arrays by rows.

The Riordan arrays are usually described by the generating
functions of their columns or, equivalently, by the induced action
on any power series. In fact a Riordan array can be defined as an
infinite matrix where the $k$-column is just the $k$-th term of a
geometric progression in $\K[[x]]$ with rate a power series of
order one. To get a proper Riordan array, eventually an element of
the Riordan group, \cite{Sha91}, we also impose that the first
term in the progression is a power series of order zero.

 In
\cite{BanPas} Section 3, the authors studied families of
polynomials associated to some particular Riordan arrays which
appeared in an iterative process to calculate the reciprocal of a
quadratic polynomial. There, we interpreted some products of
Riordan matrices as changes of variables in the associated
families of polynomials. This interpretation will be exploited
herein. Earlier in \cite{teo} the authors approached Pascal
triangle by a dynamical point of view using the Banach Fixed Point
Theorem. This approach is suitable to construct any Riordan array.
From this point of view it seems that our $T(f\mid g)$ notation
for a Riordan array is adequate, where
$\displaystyle{f=\sum_{n\geq0}f_nx^n, \ g=\sum_{n\geq0}g_nx^n}$
with $g_0\neq0$. The notation $T(f\mid g)$ represents the Riordan
array of first term $\displaystyle{\frac{f}{g}}$ and rate
$\displaystyle{\frac{x}{g}}$. So the Pascal triangle $P$ is just
$T(1\mid1-x)$. The action on a power series $h$ is given  by
$\displaystyle{T(f\mid
g)(h(x))=\frac{f(x)}{g(x)}h\left(\frac{x}{g(x)}\right)}$. The
mixture of the role of the parameters on the induced action
allowed us to get the following algorithm of construction for
$T(f\mid g)$ which is essential to get the results in this paper:
\begin{alg}\label{A:algo}
{\bf Construction of $T(f\mid g)$ }

$f=\sum_{n\geq0}f_nx^n$, $g=\sum_{n\geq0}g_nx^n$ with $g_0\neq0$,
$T(f\mid g)=(d_{n,j})$ with $n,j\geq0$,
$\frac{f}{g}=\sum_{n\geq0}d_nx^n$ and $d_{n,0}=d_n$.

\[
\left(
  \begin{array}{c|cccccccc}
f_0&&&&&&\\
f_1&d_{0,0}&d_{0,1}&d_{0,2}&d_{0,3}&d_{0,4}&\cdots\\
f_2&d_{1,0}&d_{1,1}&d_{1,2}&d_{1,3}&d_{1,4}&\cdots\\
f_3&d_{2,0}&d_{2,1}&d_{2,2}&d_{2,3}&d_{2,4}&\cdots\\
\vdots&\vdots&\vdots&\vdots&\vdots&\vdots&\cdots\\
f_{n+1}&d_{n,0}&d_{n,1}&d_{n,2}&d_{n,3}&d_{n,4}&\cdots\\
\vdots&\vdots&\vdots&\vdots&\vdots&\vdots&\ddots\\
 \end{array}
\right)
\]

with $d_{n,j}=0$ if $j>n$ and the following rules for $n\geq j$:

If $j>0$
\[
d_{n,j}=-\frac{g_1}{g_0}d_{n-1,j}-\frac{g_2}{g_0}d_{n-2,j}\cdots-\frac{g_{n}}{g_0}d_{0,j}+\frac{d_{n-1,j-1}}{g_0}
\]
and if $j=0$
\[
d_{n,0}=-\frac{g_1}{g_0}d_{n-1,0}-\frac{g_2}{g_0}d_{n-2,0}\cdots-\frac{g_{n}}{g_0}d_{0,0}+\frac{f_{n}}{g_0}
\]
Note that $d_{0,0}=\frac{f_0}{g_0}$. Then, in the $0$-column are
just the coefficients of $\frac{f}{g}$, i.e. $d_{n,0}=d_n$.
\end{alg}

The \textbf{main recurrence relation} obtained in this paper is
\begin{equation}\label{E:main}
  p_n(x)=\left(\frac{x-g_1}{g_0}\right)p_{n-1}(x)-\frac{g_2}{g_0}p_{n-2}(x)\cdots-\frac{g_{n}}{g_0}p_{0}(x)+\frac{f_{n}}{g_0}
\end{equation}

which is closely related to the algorithm. The coefficients of the
polynomials $(p_n(x))$ are, in fact, the entries in the rows of
the Riordan matrix $T(f\mid g )$.

Since our $T(f|g)$ notation for Riordan arrays is not the more
usual one, it is convenient to translate the above recurrence to
the notation $(d(t),h(t))$ with $h(0)\neq0$ and $d(0)\neq0$ used
in \cite{Huang,Spr94}. Since the rule of conversion is
$\displaystyle{(d(t),h(t))=T\left(\frac{d}{h}\Big|\frac{1}{h}\right)}$,
then the coefficients $(f_n)$ and $(g_n)$ in (\ref{E:main}) are
defined by  $\displaystyle{\frac{d}{h}=\sum_{n\geq0}f_nx^n}$ and
$\displaystyle{\frac{1}{h}=\sum_{n\geq0}g_nx^n}$. We think that
this recurrence is more difficult to predict from this last
notation.

The matrix notation used above in the algorithm will appear often
along this work so it deserves some explanation: really the matrix
$T(f\mid g)$ is what appears to the right of the vertical line.
The additional column to the left of the line, whose elements are
just the coefficients of series $f$, is needed for the
construction of the $0$-column of the matrix $T(f\mid g)$. Observe
that if we consider the whole matrix ignoring the line we get the
Riordan matrix $T(fg\mid g)$. This explanation is to avoid
repetitions along the text.

The paper is organized into four sections. In Section 2 we first
take the Pascal triangle as our first motivation. This example is
given here to explain and to motivate the interpretation of
Riordan matrices by rows. In fact, the known recurrence for
combinatorial numbers is the key to pass from the columns
interpretation to the rows interpretation and viceversa. In this
sense our Algorithm \ref{A:algo} is a huge generalization of the
rule $\displaystyle{\binom{n+1}{k}=\binom{n}{k}+\binom{n}{k-1}}$.
Later, we choose some classical sequences of polynomials:
Fibonacci, Pell and Morgan-Voyce polynomials to point out how the
structure of Riordan matrix is intrinsically in the known
recurrence relations for these families.  So we are going to
associated to any of these classical families a Riordan matrix
which determines completely the sequence of polynomials. Using the
product on the Riordan group, we recover easily some known
relationships between them.

In Section 3, we get our main recurrence relation (\ref{E:main})
as a direct consequence of Algorithm \ref{A:algo}. The theoretical
framework so constructed extends strongly and explains easily the
examples in Section 2 and some relationships between these
families. We also recover the generating function of a family of
polynomials by means of the action of $T(f\mid g)$ on a power
series. Later on, we obtain the usual umbral composition of
families of polynomials simply as a translation of the product of
matrices in the Riordan group.

In Section 4, we obtain some general recurrence relations  for any
family of generalized Appell polynomials, as a consequence of our
main recurrence (\ref{E:main}), and then of Algorithm
\ref{A:algo}. In this way we get into the so called generalized
Umbral Calculus, see \cite{Roman}, \cite{Roman-Rota}. We use the
Hadamard product of series to pass from the Riordan framework to
the more general framework of generalized Appell polynomials
because the sequences of Riordan type are those generalized Appell
sequences related to the geometric series
$\displaystyle{\frac{1}{1-x}}$, which is the neutral element for
the Hadamard product. We also relate in this section the Riordan
group with the so called delta-operators introduced by Rota et al.
\cite{Rota2}.


In this paper $\K$ always represents a field of characteristic
zero and $\N$ is the set of natural numbers including 0.

\section{Some classical examples as motivation}\label{S:ejem}

The best known description of Pascal triangle is by rows. With the
next first simple classical example we point out how to pass from
the column-description to the row-description. To do this for any
Riordan array is our main aim.

\begin{ejem}\textbf{Pascal's triangle.}
The starting point of the construction of Riordan arrays is the
Pascal triangle. From this point of view, Pascal triangle (by
columns) are the terms of the geometric progression, in $\K[[x]]$,
of first term $\displaystyle{\frac{1}{1-x}}$ and rate
$\displaystyle{\frac{x}{1-x}}$. So Pascal triangle $P$ is, by
columns,
$\displaystyle{P=\left(\frac{1}{1-x},\frac{x}{(1-x)^2},\frac{x^2}{(1-x)^3},\cdots,\frac{x^n}{(1-x)^{n+1}},\cdots,\right)}$.
Of course  it is not the way to introduce Pascal triangle, or
Tartaglia triangle, for the first time to students, because in
particular it requires some understanding of the abstraction of
infinity and order both on the \emph{number} of columns and on the
elements in any column. On the contrary, the non-null elements in
any row of Pascal triangle form a finite set of data. Usually
Pascal triangle is introduce by rows as the coefficients of the
sequence of polynomials $p_n(x)=(x+1)^n$. So, by rows, Pascal
triangle is
\[
P=\left(%
\begin{array}{c}
  (x+1)^0 \\
   (x+1)^1\\
  (x+1)^2 \\
  \vdots \\
  (x+1)^n \\
   \vdots\\
\end{array}%
\right)
\]
The Newton formula
$\displaystyle{(x+1)^n=\sum_{k=0}^n\binom{n}{k}x^k}$ allows us to
say that the $n$-th row of Pascal triangle is, by increasing order
of power of $x$,
$\displaystyle{\binom{n}{0},\binom{n}{1},\binom{n}{2},\cdots,\binom{n}{n}}$.
As it is well-known, $\displaystyle{\binom{n}{k}}$ represents the
number of subsets, with exactly $k$-elements, of a set with $n$
elements. Using algebra, $(x+1)^{n+1}=(x+1)(x+1)^n$, or
combinatorics, counting subsets, we see that
$\displaystyle{\binom{n+1}{k}=\binom{n}{k}+\binom{n}{k-1}}$. This
means that the Pascal triangle $P=(p_{n,k})_{n,k\in\N}$ follows
the rule: $p_{n,0}=1$ for every $n\in\N$, because
$\displaystyle{\binom{n}{0}=1}$ and $p_{n+1,k}=p_{n,k}+p_{n,k-1}$
for $1\leq k\leq n$. Using for example the combinatorial
interpretation of $\displaystyle{\binom{n}{k}}$ we see at once
that $\displaystyle{\binom{n}{k}=0}$ if $k>n$. 
What is the same, the Pascal triangle $(p_{n,k})_{n,k\in\N}$ is
totally determined by the following recurrence relation: If we
consider $\displaystyle{p_n(x)=\sum_{k=0}^np_{n,k}x^k}$ then
$p_0(x)=1$ and $p_{n+1}(x)=(x+1)p_{n}(x)$, $\forall\ n\geq0$. It
is  obvious because the above relations means that
$p_n(x)=(x+1)^n$.
\end{ejem}

\begin{ejem} \textbf{The Fibonacci polynomials. The Pell
polynomials. The Morgan-Voyce polynomials.} 
The Fibonacci polynomials are the polynomials defined by,
$F_0(x)=1$, $F_1(x)=x$ and
\[
F_{n}(x)=xF_{n-1}(x)+F_{n-2}(x) \ \text{ for } \ n\geq2
\]
We can unify the recurrence relation with the initial conditions
if we consider the sequence $(f_n)_{n\in\N}$, $(g_n)_{n\in\N}$
given by $g_0=1, \ g_1=0, \ g_2=-1, \ g_n=0, \ \forall n\geq3$ and
$f_0=1, \ f_n=0 \ \forall n\geq1$. Because if we write
\[
F_n(x)=\left(\frac{x-g_1}{g_0}\right)F_{n-1}(x)-\frac{g_2}{g_0}F_{n-2}(x)
-\cdots-\frac{g_n}{g_0}F_0(x)+\frac{f_n}{g_0}
\]
For $n\geq0$ we obtain both the recurrence relation and the
initial conditions. Note that the above recurrence for Fibonacci
polynomials fits the main recurrence relation (\ref{E:main}).

If we consider the Riordan matrix $T(f\mid g)$ for $f=1$ and
$g=1-x^2$, $\displaystyle{T\left(1|1-x^2\right)=(d_{n,k})}$ then
the polynomials associated to $T\left(1|1-x^2\right)$ are just the
Fibonacci polynomials. Using Algorithm \ref{A:algo}, the rule of
construction is: $d_{n,k}=d_{n-2,k}+d_{n-1,k-1}$, for $k>0$ and
$d_{n,0}=d_{n-2,0}$ for $n\geq2$ and $d_{0,0}=1$ and $d_{1,0}=0$.
The few first rows are:
\[
\left(
  \begin{array}{c|cccccccc}
1&&&&&&\\
0&1&&&&&\\
0&0&1&&&&\\
0&1&0&1&&&\\
0&0&2&0&1&&\\
0&1&0&3&0&1&\\
0&0&3&0&4&0&1\\
0&1&0&6&0&5&0&1\\
\vdots&\vdots&\vdots&\vdots&\vdots&\vdots&\vdots&\vdots&\ddots\\
  \end{array}
\right)
\]
Consequently the few first associated polynomials (look at the
rows of the matrix) are

\noindent $F_0(x)=1$\\
\noindent $F_1(x)=x$\\
\noindent $F_2(x)=1+x^2$\\
\noindent $F_3(x)=2x+x^3$\\
\noindent $F_4(x)=1+3x^2+x^4$\\
\noindent $F_5(x)=3x+4x^3+x^5$\\
\noindent $F_6(x)=1+6x^2+5x^4+x^6$\\
Which are the Fibonacci polynomials. Using the induced action of
$T\left(1|1-x^2\right)$ we get the generating function of this
sequence
\[
\sum_{n\geq0}F_n(t)x^n=T\left(1|1-x^2\right)\left(\frac{1}{1-xt}\right)=\frac{1}{1-x^2-xt}
\]

The Pell polynomials are related to the Fibonacci polynomials. Now
we consider   $P_0(x)=1$ and $P_1(x)=2x$ with the polynomial
recurrence $P_{n}(x)=2xP_{n-1}(x)+P_{n-2}(x)$. So
$\displaystyle{\frac{x-g_1}{g_0}=2x, \ \frac{-g_2}{g_0}=1}$ then
$\displaystyle{g(x)=\frac{1}{2}-\frac{1}{2}x^2}$ and
$\displaystyle{f(x)=\frac{1}{2}}$. Hence the Riordan matrix
involved is
$\displaystyle{T\left(\frac{1}{2}\Big|\frac{1}{2}-\frac{1}{2}x^2\right)}$
with the rule of construction: \[d_{n,k}=d_{n-2,k}+2d_{n-1,k-1},
\qquad k>0\] again the few first rows are:
\[
\left(
  \begin{array}{c|cccccccc}
\frac{1}{2}&&&&&&\\
0&1&&&&&\\
0&0&2&&&&\\
0&1&0&4&&&\\
0&0&4&0&8&&\\
0&1&0&12&0&16&\\
0&0&6&0&32&0&32\\
0&1&0&24&0&80&0&64\\
\vdots&\vdots&\vdots&\vdots&\vdots&\vdots&\vdots&\vdots&\ddots\\
  \end{array}
\right)
\]

with generating function
\[
\sum_{n\geq0}P_n(t)x^n=T\left(\frac{1}{2}|\frac{1}{2}-\frac{1}{2}x^2\right)\left(\frac{1}{1-xt}\right)=\frac{1}{1-x^2-2xt}
\]
We note that:
\[
T\left(\frac{1}{2}\Big|1\right)T(1|1-x^2)T\left(1\Big|\frac{1}{2}\right)=T\left(\frac{1}{2}\Big|\frac{1}{2}-\frac{1}{2}x^2\right)
\]
So, following Proposition 14 in \cite{BanPas}, we get that
$P_n(x)=F_n(2x)$ that is a known property of Pell polynomials.

Another related families of polynomials that we can treat using
these techniques are the Morgan-Voyce families polynomials. If we
consider now the Riordan matrices $T(1|(1-x)^2)$ and
$T(1-x|(1-x)^2)$. These triangles have the same rule of
construction
$\displaystyle{d_{n,k}=2d_{n-1,k}-2d_{n-2,k}+d_{n-1,k-1}}$ but
different initial condition. In fact they are:
\[
\left(
  \begin{array}{c|cccccccc}
1&&&&&&\\
0&1&&&&&\\
0&2&1&&&&\\
0&3&4&1&&&\\
0&4&10&6&1&&\\
0&5&20&21&8&1&\\
0&6&35&56&36&10&1\\
0&7&56&126&120&55&12&1\\
\vdots&\vdots&\vdots&\vdots&\vdots&\vdots&\vdots&\vdots&\ddots\\
  \end{array}
\right) \qquad \left(
  \begin{array}{c|cccccccc}
1&&&&&&\\
-1&1&&&&&\\
0&1&1&&&&\\
0&1&3&1&&&\\
0&1&6&5&1&&\\
0&1&10&15&7&1&\\
0&1&15&35&28&9&1\\
0&1&21&70&84&45&11&1\\
\vdots&\vdots&\vdots&\vdots&\vdots&\vdots&\vdots&\vdots&\ddots\\
  \end{array}
\right)
\]
where

\noindent $B_0(x)=1\qquad\qquad\qquad\qquad\qquad\qquad\qquad\qquad\qquad    b_0(x)=1$\\
\noindent $B_1(x)=2+x\qquad\qquad\qquad\qquad\qquad\qquad\qquad\qquad\     b_1(x)=1+x$\\
\noindent $B_2(x)=3+4x+x^2\qquad\qquad\qquad\qquad\qquad\qquad\quad \  b_2(x)=1+3x+x^2$\\
\noindent $B_3(x)=4+10x+6x^2+x^3\qquad\qquad\qquad\qquad\qquad  b_3(x)=1+6x+5x^2+x^3$\\

In general

\noindent $B_n(x)=(x+2)B_{n-1}(x)-B_{n-2}(x)\qquad\qquad\qquad b_n(x)=(x+2)b_{n-1}(x)-b_{n-2}(x)$\\

with generating functions:
\[
\sum_{n\geq0}B_n(t)x^n=T(1|(1-x)^2)\left(\frac{1}{1-xt}\right)=\frac{1}{1-(2+t)x+x^2}
\]
\[
\sum_{n\geq0}b_n(t)x^n=T(1-x|(1-x)^2)\left(\frac{1}{1-xt}\right)=\frac{1-x}{1-(2+t)x+x^2}
\]


On the other hand it is known that the sequences
 $(B_n(x))_{n\in\N}$ and $(b_n(x))_{n\in\N}$ are related by means of the equalities:
\[
B_n(x)=(x+1)B_{n-1}(x)+b_{n-1}(x)\]\[ b_n(x)=xB_{n-1}(x)+b_{n-1}(x)
\]
Or equivalently
\begin{equation}\label{E:M-V 1}
    B_n(x)-B_{n-1}(x)=b_{n}(x)
\end{equation}

\begin{equation}\label{E:M-V 2}
 b_n(x)-b_{n-1}(x)=xB_{n-1}(x)
\end{equation}

These equalities can be interpreted by means of the product of
adequate Riordan arrays. The first of them,  $(\ref{E:M-V 1})$, is
\[
T(1-x|1)T(1|(1-x)^2)=T(1-x|(1-x)^2)
\]
or,
\[
\left(
  \begin{array}{cccccccc}
1&&&&&\\
-1&1&&&&\\
0&-1&1&&&\\
0&0&-1&1&&\\
0&0&0&-1&1&\\
\vdots&\vdots&\vdots&\vdots&\vdots&\ddots\\
  \end{array}
\right) \left(
  \begin{array}{cccccccc}
1&&&&&\\
2&1&&&&\\
3&4&1&&&\\
4&10&6&1&&\\
5&20&21&8&1&\\
\vdots&\vdots&\vdots&\vdots&\vdots&\ddots\\
  \end{array}
\right)= \left(
  \begin{array}{cccccccc}
1&&&&&\\
1&1&&&&\\
1&3&1&&&\\
1&6&5&1&&\\
1&10&15&7&1&\\
\vdots&\vdots&\vdots&\vdots&\vdots&\ddots\\
  \end{array}
\right)
\]
For the equality $(\ref{E:M-V 2})$ we consider the product of
matrices
\[
T(1-x|1)T(1-x|(1-x)^2)=T((1-x)^2|(1-x)^2)
\]
or,
\[
\left(
  \begin{array}{cccccccc}
1&&&&&\\
-1&1&&&&\\
0&-1&1&&&\\
0&0&-1&1&&\\
0&0&0&-1&1&\\
\vdots&\vdots&\vdots&\vdots&\vdots&\ddots\\
  \end{array}
\right) \left(
  \begin{array}{cccccccc}
1&&&&&\\
1&1&&&&\\
1&3&1&&&\\
1&6&5&1&&\\
1&10&15&7&1&\\
\vdots&\vdots&\vdots&\vdots&\vdots&\ddots\\
  \end{array}
\right)
 =\left(
  \begin{array}{ccccccccc}
1&&&&&\\
0&1&&&&\\
0&2&1&&&\\
0&3&4&1&&\\
0&4&10&6&1&\\
\vdots&\vdots&\vdots&\vdots&\vdots&\ddots\\
  \end{array}
\right)
\]

\end{ejem}

\section{Polynomial sequences associated to Riordan matrices and its recurrence
relations}\label{S:main}

In this section we are going to obtain the basic main result in
this paper as a consequence of our algorithm in \cite{teo} and
stated again in the Introduction as Algorithm \ref{A:algo}. We use
\cite{teo} and \cite{BanPas} for notation and basic results.

\subsection{The main theorem}
\begin{defi}
Consider an infinite lower triangular matrix
$A=(a_{n,j})_{n,j\in\N}$.  We define \textit{the family of
polynomials associated to $A$ }, to the sequence of polynomials
$(p_n(x))_{n\in\N}$, given by
\[p_n(x)=\sum_{j=0}^na_{n,j}x^j, \quad \text{with}\quad n\in\N\]
\end{defi}

Note that the coefficients of the polynomials are given by the
entries in the rows of $A$ in increasing order of the columns till
the main diagonal. Note also that the degree of $p_n(x)$ is less
than or equal to $n$. The family $p_n(x)$ becomes a
\emph{polynomial sequences}, in the usual sense, when the matrix
$A$ is invertible, that is, when all the elements in the main
diagonal are non-null.

Our main result can be given in the following terms:

\begin{thm}\label{T:recu}
Let $D=(d_{n,j})_{n,j\in\N}$  be an infinite lower triangular
matrix. $D$ is a Riordan matrix, or an arithmetical triangle in
the sense of \cite{teo}, if and only if there exist two sequences
$(f_n)$ and $(g_n)$ in $\K$ with $g_0\neq0$ such that the family
of polynomials associated to $D$ satisfies the recurrence
relation:
\[
p_n(x)=\left(\frac{x-g_1}{g_0}\right)p_{n-1}(x)-\frac{g_2}{g_0}p_{n-2}(x)\cdots-\frac{g_{n}}{g_0}p_{0}(x)+\frac{f_{n}}{g_0}
\qquad \forall n\geq0\]
Moreover, in this case, $D=T(f\mid g)$
where $f=\sum_{n\geq0}f_nx^n$ and $g=\sum_{n\geq0}g_nx^n$.
\end{thm}

\begin{proof}
If $D$ is a Riordan array we can identify this with an arithmetical
triangle $D=T(f\mid g)$ such that $g_0\neq0$. Following Algorithm
\ref{A:algo} we obtain that the family of polynomials associated to
$T(f\mid g)$ satisfies:
\[
p_n(x)=\sum_{j=0}^nd_{n,j}x^j=d_{n,0}+\sum_{j=1}^nd_{n,j}x^j=\]
\[=\frac{1}{g_0}\left(f_{n}-\sum_{k=1}^{n}g_kd_{n-k,0}\right)+
\sum_{j=1}^n\left(\frac{1}{g_0}\left(d_{n-1,j-1}-\sum_{k=1}^{n}g_kd_{n-k,j}\right)
\right)x^j=
\]
\[
=\frac{1}{g_0}\left(f_{n}-\sum_{j=1}^nd_{n-1,j-1}x^j-\sum_{k=1}^{n}g_kd_{n-k,0}-
\sum_{j=1}^n\sum_{k=1}^{n}g_kd_{n-k,j}x^j\right)=
\]
\[
=\frac{1}{g_0}\left(f_{n}-xp_{n-1}(x)-
\sum_{j=0}^n\sum_{k=1}^{n}g_kd_{n-k,j}x^j\right)=\frac{1}{g_0}\left(f_{n}-xp_{n-1}(x)-
\sum_{k=1}^ng_k\sum_{j=0}^{n-k}d_{n-k,j}x^j\right)=
\]
\[
\frac{1}{g_0}\left(f_{n}-xp_{n-1}(x)-
\sum_{k=1}^ng_kp_{n-k}(x)\right)=\frac{1}{g_0}\left(f_{n}+(g_1-x)p_{n-1}(x)-
\sum_{k=2}^ng_kp_{n-k}(x)\right)=
\]
\[
=\left(\frac{x-g_1}{g_0}\right)p_{n-1}(x)-\frac{g_2}{g_0}p_{n-2}(x)\cdots-\frac{g_{n}}{g_0}p_{0}(x)+\frac{f_{n}}{g_0}
\]

On the other hand, we suppose that
\[
p_n(x)=\left(\frac{x-g_1}{g_0}\right)p_{n-1}(x)-\frac{g_2}{g_0}p_{n-2}(x)\cdots-\frac{g_{n}}{g_0}p_{0}(x)+\frac{f_{n}}{g_0}
\]
for two sequences $(f_n)$ and $(g_n)$. We consider $D=(d_{n,k})$
such that $\displaystyle{p_n(x)=\sum_{j=0}^nd_{n,j}x^j}$. So
$\displaystyle{p_0(x)=\frac{f_0}{g_0}}$ then
$\displaystyle{d_{0,0}=\frac{f_0}{g_0}}$.
\[
p_1(x)=\left(\frac{x-g_1}{g_0}\right)p_0(x)+\frac{f_1}{g_0}=-\frac{g_1}{g_0}d_{0,0}+\frac{f_1}{g_0}+\frac{d_{0,0}}{g_0}x
\]
then
\[
d_{1,0}=-\frac{g_1}{g_0}d_{0,0}+\frac{f_1}{g_0},\qquad
d_{1,1}=\frac{d_{0,0}}{g_0}
\]
\[
p_2(x)=\left(\frac{x-g_1}{g_0}\right)p_1(x)-\frac{g_2}{g_0}p_0(x)+\frac{f_2}{g_0}=
-\frac{g_1}{g_0}d_{1,0}-\frac{g_2}{g_0}d_{0,0}+\frac{f_2}{g_0}+\left(-\frac{g_1}{g_0}d_{1,1}+\frac{d_{1,0}}{g_0}\right)x+\frac{d_{1,1}}{g_0}x^2
\]
so
\[
d_{2,0}=-\frac{g_1}{g_0}d_{1,0}-\frac{g_2}{g_0}d_{0,0}, \qquad
d_{2,1}=-\frac{g_1}{g_0}d_{1,1}+\frac{d_{1,0}}{g_0}, \qquad
d_{2,2}=\frac{d_{1,1}}{g_0}
\]
in general
\[
p_n(x)=\left(\frac{x-g_1}{g_0}\right)p_{n-1}(x)-\frac{g_2}{g_0}p_{n-2}(x)\cdots-\frac{g_{n}}{g_0}p_{0}(x)+\frac{f_{n}}{g_0}
\]
then
\[
d_{n,0}=-\frac{g_1}{g_0}d_{n-1,0}-\frac{g_2}{g_0}d_{n-2,0}\cdots-\frac{g_{n}}{g_0}d_{0,0}+\frac{f_{n}}{g_0}
\]
\[
d_{n,1}=-\frac{g_1}{g_0}d_{n-1,1}-\frac{g_2}{g_0}d_{n-2,1}\cdots-\frac{g_{n}}{g_0}d_{0,1}+\frac{d_{n-1,0}}{g_0}
\]
\[
d_{n,j}=-\frac{g_1}{g_0}d_{n-1,j}-\frac{g_2}{g_0}d_{n-2,j}\cdots-\frac{g_{n}}{g_0}d_{0,j}+\frac{d_{n-1,j-1}}{g_0}
\]
and
\[
d_{n,n-1}=-\frac{g_1}{g_0}d_{n-1,n-1}+\frac{d_{n-1,n-2}}{g_0},
\qquad d_{n,n}=\frac{d_{n-1,n-1}}{g_0}
\]
then using our algorithm the matrix $D$ is just $D=T(f|g)$ where
$f(x)=\sum_{n\geq0}f_nx^n$ and $g(x)=\sum_{n\geq0}g_nx^n$.
\end{proof}

\begin{cor}
If $g(x)=g_0+g_1x+g_2x^2+\cdots+g_mx^m$ with $g_m\neq0$ be a
polynomial of degree $m$, the recurrence relation of Theorem
\ref{T:recu} is eventually finite. It is,
\[
p_n(x)=\left(\frac{x-g_1}{g_0}\right)p_{n-1}(x)-\frac{g_2}{g_0}p_{n-2}(x)\cdots-\frac{g_{m}}{g_0}p_{n-m}(x)+\frac{f_{n}}{g_0}
\qquad n\geq m
\]
and
\[
p_k(x)=\left(\frac{x-g_1}{g_0}\right)p_{k-1}(x)-\sum_{i=2}^k\frac{g_i}{g_0}p_{k-i}(x)+\frac{f_k}{g_0}
\qquad 0\leq k \leq m-1
\]
\end{cor}

\begin{rmk}\label{R:f0neq0}
Following \cite{teo} the arithmetical triangle $T(f\mid g)$ above
is an element of the Riordan group when it is invertible for the
product of matrices. It is obviously equivalent to the fact that
$f_0\neq0$ in the sequence $(f_n)$ above.
\end{rmk}

Suppose that we have two Riordan matrices $T(f|g)$, $T(l|m)$ with
$\displaystyle{f=\sum_{n\geq0}f_nx^n,\ g=\sum_{n\geq0}g_nx^n}$
$\displaystyle{l=\sum_{n\geq0}l_nx^n}$ and
$\displaystyle{m=\sum_{n\geq0}m_nx^n}$ with $g_0, m_0\neq0$.
Consider the corresponding families of polynomials
$(p_n(x))_{n\in\N}$ and $(q_n(x))_{n\in\N}$ associated to $T(f|g)$
and  $T(l|m)$ respectively, as in Theorem \ref{T:recu}. Using the
matrix representation of $T(f|g)$ and $T(l|m)$, \cite{teo}, and
the product of matrices, we can define an operation $\sharp$ on
these sequences of polynomials as follows:

We say that
\[(p_n(x))_{n\in\N}\sharp(q_n(x))_{n\in\N}=(r_n(x))_{n\in\N}\]
where $(r_n(x))_{n\in\N}$ is the family of polynomials associated
to the Riordan matrix
\[T(f|g)T(l|m)=T\left(fl\left(\frac{x}{g}\right)\Big|gm\left(\frac{x}{g}\right)\right)\]
see \cite{teo}.

Suppose $T(f|g)=(p_{n,k})_{n,k\in\N}$, $T(l|m)=(q_{n,k})_{n,k\in\N}$
and
$T\left(fl\left(\frac{x}{g}\right)\Big|gm\left(\frac{x}{g}\right)\right)=(r_{n,k})_{n,k\in\N}$.
Consequently $\displaystyle{p_n(x)=\sum_{k=0}^np_{n,k}x^k}$,
$\displaystyle{q_n(x)=\sum_{k=0}^nq_{n,k}x^k}$ and
$\displaystyle{r_n(x)=\sum_{k=0}^nr_{n,k}x^k}$.

\[
\left(
  \begin{array}{cccccc}
    p_{0,0} &  &  &  & & \\
    p_{1,0} & p_{1,1} &  &  & & \\
    p_{2,0} & p_{2,1} & p_{2,2} &  & & \\
    \vdots & \vdots & \vdots & \ddots &&  \\
    p_{n,0} & p_{n,1} & p_{n,2} & \cdots & p_{n,n}& \cdots\\
    \vdots & \vdots & \vdots & \cdots &\vdots& \ddots  \\
  \end{array}
\right) \left(
  \begin{array}{cccccc}
    q_{0,0} &  &  &  & & \\
    q_{1,0} & q_{1,1} &  &  & & \\
    q_{2,0} & q_{2,1} & q_{2,2} &  & & \\
    \vdots & \vdots & \vdots & \ddots &&  \\
    q_{n,0} & q_{n,1} & q_{n,2} & \cdots & q_{n,n}& \cdots\\
    \vdots & \vdots & \vdots & \cdots &\vdots& \ddots  \\
  \end{array}
\right)= \left(
  \begin{array}{cccccc}
    r_{0,0} &  &  &  & & \\
    r_{1,0} & r_{1,1} &  &  & & \\
    r_{2,0} & r_{2,1} & r_{2,2} &  & & \\
    \vdots & \vdots & \vdots & \ddots &&  \\
    r_{n,0} & r_{n,1} & r_{n,2} & \cdots & r_{n,n}& \cdots\\
    \vdots & \vdots & \vdots & \cdots &\vdots& \ddots  \\
  \end{array}
\right)
\]

So the entries in the $n$-row of $(r_{n,k})$, which are just the
coefficients of $r_n(x)$ in increasing order of the power of $x$,
are given by:

\[
\left(\sum_{k=0}^np_{n,k}q_{k,0},\sum_{k=1}^np_{n,k}q_{k,1},
\cdots\sum_{k=j}^np_{n,k}q_{k,j}\cdots
p_{n,n}q_{n,n},0,\cdots\right)=\]
\[p_{n,0}(q_{0,0},0,\cdots,0,\cdots)+p_{n,1}(q_{1,0},q_{1,1},0,\cdots,0,\cdots)+
\cdots+p_{n,n}(q_{n,0},q_{n,1},\cdots,q_{n,n},0,\cdots)
\]
Consequently
\[
r_n(x)=\sum_{k=0}^np_{n,k}q_k(x)
\]
which corresponds to substitute in the expression of
$p_n(x)=\sum_{k=0}^np_{n,k}x^k$ the power $x^k$ by the element
$q_k(x)$ in the sequence of polynomials $(q_{n}(x))_{n\in\N}$.
This is in the spirit of the Blissard symbolic's method, see
\cite{Bell} for an exposition on this topic. The product
$\displaystyle{(p_n(x))_{n\in\N}\sharp(q_n(x))_{n\in\N}=(r_n(x))_{n\in\N}}$
is usually called the umbral composition of the sequences of
polynomials $(p_n(x))$ and $(q_n(x))$. The formula for the umbral
composition is given by
\[
(p_n(x))_{n\in\N}\sharp(q_n(x))_{n\in\N}=(r_n(x))_{n\in\N}
\]
where
\[
r_{n,j}=\sum_{k=j}^np_{n,k}q_{k,j}
\]
As a summary of the above construction we have:
\begin{thm}
Suppose four sequences of elements of $\K$, $(f_n)_{n\in\N}$,
$(g_n)_{n\in\N}$, $(l_n)_{n\in\N}$, $(m_n)_{n\in\N}$, with
$g_0,m_0\neq0$. Consider the sequences of polynomials
$(p_n(x))_{n\in\N}$ $(q_n(x))_{n\in\N}$ satisfying the following
recurrences relations
\[
p_n(x)=\left(\frac{x-g_1}{g_0}\right)p_{n-1}(x)-\frac{g_2}{g_0}p_{n-2}(x)\cdots-\frac{g_{n}}{g_0}p_{0}(x)+\frac{f_{n}}{g_0}
\]
with $\displaystyle{p_0(x)=\frac{f_0}{g_0}}$,
\[
q_n(x)=\left(\frac{x-m_1}{m_0}\right)q_{n-1}(x)-\frac{m_2}{m_0}q_{n-2}(x)\cdots-\frac{m_{n}}{m_0}q_{0}(x)+\frac{l_{n}}{m_0}
\]
with $\displaystyle{q_0(x)=\frac{l_0}{m_0}}$. Then the umbral
composition
$\displaystyle{(p_n(x))_{n\in\N}\sharp(q_n(x))_{n\in\N}=(r_n(x))_{n\in\N}}$
satisfies the following recurrence relation
\[
r_n(x)=\left(\frac{x-\al_1}{\al_0}\right)r_{n-1}(x)-\frac{\al_2}{\al_0}r_{n-2}(x)\cdots-\frac{\al_{n}}{\al_0}r_{0}(x)+\frac{\beta_{n}}{\al_0}
\]
where $(\al_n)_{n\in\N}$, $(\beta_n)_{n\in\N}$ are sequences such
that
$\displaystyle{fl\left(\frac{x}{g}\right)=\sum_{n\geq0}\beta_nx^n}$,
$\displaystyle{gm\left(\frac{x}{g}\right)=\sum_{n\geq0}\al_nx^n}$,
with $\displaystyle{f=\sum_{n\geq0}f_nx^n}$,
$\displaystyle{g=\sum_{n\geq0}g_nx^n}$
$\displaystyle{l=\sum_{n\geq0}l_nx^n}$ and
$\displaystyle{m=\sum_{n\geq0}m_nx^n}$.
\end{thm}

Of special interest is when we restrict ourselves to the so called
proper Riordan arrays, see \cite{Spr94}. As noted in Remark
\ref{R:f0neq0} this is the case when $f_0\neq0$ or, equivalently,
$T(f\mid g)$ is in the Riordan group. Moreover, in this case, the
assignment $T(f|g)\rightarrow(p_n(x))_{n\in\N}$ is injective,
obviously, and since the product of matrices converts to the
umbral composition of the corresponding associated polynomial
sequences, we have the following alternative description of the
Riordan group.
\begin{thm}
Let $\K$ be a field of characteristic zero. Consider
$\mathcal{R}=\{(p_n(x))_{n\in\N}\}$ where $(p_n(x))_{n\in\N}$ is a
polynomial sequence with coefficients in $\K$ satisfying that
there are two sequences $(f_n)_{n\in\N}$, $(g_n)_{n\in\N}$ of
elements of $\K$, depending on $(p_n(x))_{n\in\N}$, with
$f_0,g_0\neq0$ and such that
\[
p_n(x)=\left(\frac{x-g_1}{g_0}\right)p_{n-1}(x)-\frac{g_2}{g_0}p_{n-2}(x)\cdots-\frac{g_{n}}{g_0}p_{0}(x)+\frac{f_{n}}{g_0}
\]
with $\displaystyle{p_0(x)=\frac{f_0}{g_0}}$.

Given $(p_n(x))_{n\in\N}$, $(q_n(x))_{n\in\N}$ $\in \mathcal{R}$
Define
$\displaystyle{(p_n(x))_{n\in\N}\sharp(q_n(x))_{n\in\N}=(r_n(x))_{n\in\N}}$
where $r_n(x)=\sum_{k=0}^np_{n,k}q_k(x)$ with
$p_n(x)=\sum_{k=0}^np_{n,k}x^k$. Then $(\mathcal{R},\sharp)$ is a
group isomorphic to the Riordan group. Moreover
\[
\sum_{n\geq0}p_n(t)x^n=\frac{f(x)}{g(x)-xt}
\]
if $\displaystyle{f=\sum_{n\geq0}f_nx^n}$ and
$\displaystyle{g=\sum_{n\geq0}g_nx^n}$ and $(f_n)$ and $(g_n)$ are
the sequences generating the polynomial sequence $(p_n(x))$ in
$\mathcal{ R}$.
\end{thm}
\begin{proof}
Only a proof of the final part is needed. As we know, from Theorem
\ref{T:recu}, $T(f|g)=(p_{n,k})_{n,k\in\N}$ is a proper Riordan
array where $p_n(x)=\sum_{k=0}^np_{n,k}x^k$,
$\displaystyle{\frac{1}{1-xt}=\sum_{n\geq0}t^nx^n}$. We consider,
symbolically, $\displaystyle{\frac{1}{1-xt}}$ as a power series on
$x$ with \emph{parametric} coefficients $a_n=t^n$. From this point
of view, \cite{teo},
\[
T(f|g)\left(\frac{1}{1-xt}\right)=\left(
  \begin{array}{cccccc}
    p_{0,0} &  &  &  & & \\
    p_{1,0} & p_{1,1} &  &  & & \\
    p_{2,0} & p_{2,1} & p_{2,2} &  & & \\
    \vdots & \vdots & \vdots & \ddots &&  \\
    p_{n,0} & p_{n,1} & p_{n,2} & \cdots & p_{n,n}& \cdots\\
    \vdots & \vdots & \vdots & \cdots &\vdots& \ddots  \\
  \end{array}
\right) \left(
 \begin{array}{c}
 1\\
 t\\
 t^2\\
 \vdots\\
 t^n\\
 \vdots
  \end{array}
 \right)=\sum_{k=0}^np_n(t)x^k
\]
\[
T(f|g)\left(\frac{1}{1-xt}\right)=\frac{f(x)}{g(x)}\frac{1}{1-t\frac{x}{g}}=\frac{f(x)}{g(x)-xt}
\]
\end{proof}

\begin{rmk}
Note that $\displaystyle{\sum_{k=0}^np_n(t)x^k}$ is just the
bivariate generating function of the Riordan array
$T(f|g)=(p_{n,k})_{n,k\in\N}$ in the sense of \cite{Spr94}.
\end{rmk}

\subsection{Some relationships between polynomials sequences of Riordan type and some classical examples}

Now we are going to describe some relations between polynomial
sequences associated to different but related Riordan arrays. From
now on we are going to use the following definition:
\begin{defi}
Let $(p_n(x))_{n\in\N}$ be a sequence of polynomials in $\K[[x]]$,
$p_n(x)=\sum_{k=0}^np_{n,k}x^k$. We say that $(p_n(x))_{n\in\N}$
is a \emph{polynomial sequence of Riordan type} if the matrix
$(p_{n,k})$ is an element of the Riordan group.
\end{defi}

Using the basic equality $\displaystyle{T(f\mid g )=T(f\mid 1
)T(1\mid g )}$ we can get some formulas.

\begin{prop}\label{P:relat1pn}
Let $T(f\mid g)$ an element of the Riordan group and suppose
$(p_n(x))$ the corresponding associated family of polynomials. Let
$h(x)=h_0+h_1x+h_2x^2+\cdots+h_mx^m$ be a $m$ degree polynomial,
$h_m\neq0$. Let $(q_n(x))$ be the associated family of polynomials
of $T(h\mid 1)T(f\mid g)$ then
\[
q_0(x)=h_0p_0(x)
\]
\[
q_1(x)=h_1p_0(x)+h_0p_1(x)
\]
\[
\vdots
\]
\[
q_m(x)=h_mp_{n-m}(x)+\cdots+h_0p_m(x)
\]
\[
q_n(x)=h_mp_{n-m}(x)+\cdots+h_0p_n(x)\ \qquad n\geq m
\]
\end{prop}

\begin{rmk}
Note that to multiply by the left by the Toepliz matrix
$T(h\mid1)$ above corresponds eventually to make some fixed
elementary operations by rows on the matrix $T(f\mid g)$. These
operations are completely determined by the coefficients of the
polynomial $h$. For example if $h(x)=a+bx$ then $q_0(x)=ap_0(x)$
and $q_n(x)=bp_{n-1}(x)+ap_n(x)$.
\end{rmk}

As a direct application of Proposition \ref{P:relat1pn} we will
obtain the known relationships between Chebysev polynomials of the
first and second kind.

\begin{ejem} {\bf The Chebyshev polynomials of the first and the second
kind.}

Consider the Chebyshev polynomials of the second kind:

\begin{equation}\label{E:U}
\begin{array}{c}
                                    U_0(x)=1 \\
                                    U_1(x)=2x \\
                                    U_2(x)=4x^2-1 \\
                                    U_3(x)=8x^3-4x \\
                                    U_4(x)=16x^4-12x^2+1\\
U_n(x)=2xU_{n-1}(x)-U_{n-2}(x) \qquad \text{ for } \qquad n\geq2
                                  \end{array}
\end{equation}

Let the sequences $(l_n)_{n\in\N}$, $(m_n)_{n\in\N}$ given by
$\displaystyle{l_0=\frac{1}{2}}$ and $l_n=0$ for $n\geq1$ and
$\displaystyle{m_0=\frac{1}{2}}$, $\displaystyle{m_2=\frac{1}{2}}$
and $m_n=0$ otherwise. In this case $(\ref{E:U})$ can be converted
to

\begin{equation}\label{E:Urecurrence}
\begin{array}{c}
  U_0(x)=\frac{l_0}{m_0} \\
  U_n(x)=\left(\frac{x-m_1}{m_0}\right)U_{n-1}(x)-\frac{m_2}{m_0}U_{n-2}(x)\cdots-\frac{m_{n}}{m_0}U_{0}(x)+\frac{l_{n}}{m_0},
  \ \text{for} \ n\geq1
\end{array}
\end{equation}

If $U=(u_{n,k})_{n,k\in\N}$ where
$\displaystyle{U_n(x)=\sum_{k=0}^nu_{n,k}x^k}$ then using our
algorithm, or equivalently Theorem \ref{T:recu}, we obtain that
$U=\displaystyle{T\left(\frac{1}{2}\Big|\frac{1}{2}+\frac{1}{2}x^2\right)}$
is a Riordan matrix:
\[
\left(
  \begin{array}{c|cccccc}
\frac{1}{2}&&&&&&\\
0&1&&&&&\\
0&0&2&&&&\\
0&-1&0&4&&&\\
0&0&-4&0&8&&\\
0&1&0&-12&0&16&\\
\vdots&\vdots&\vdots&\vdots&\vdots&\vdots&\ddots\\
  \end{array}
\right)
\]

So the associated polynomials of this arithmetical triangle are
the Chebyshev polynomials of the second kind.

Consequently
\[
\sum_{n\geq0}U_n(t)x^n=T\left(\frac{1}{2}\Big|\frac{1}{2}+\frac{1}{2}x^2\right)\left(\frac{1}{1-xt}\right)=\frac{1}{1+x^2-2xt}
\]

The first few Chebyshev polynomials of the first kind are

\noindent $T_0(x)=1$\\
\noindent $T_1(x)=x$\\
\noindent $T_2(x)=2x^2-1$\\
\noindent $T_3(x)=4x^3-3x$\\
\noindent $T_4(x)=8x^4-8x^2+1$\\

In general
\[
T_n(x)=2xT_{n-1}(x)-T_{n-2}(x) \qquad \text{ for } \qquad n\geq2
\]

We first produce a small perturbation in this classical sequence.
Consider a new sequence $(\tT(x))_{n\in\N}$ where
$\displaystyle{\tT_0(x)=\frac{1}{2}}$ and $\tT_n(x)=T_n(x)$ for
every $n\geq1$. For this new sequence we have the following
recurrence relation
\begin{equation}\label{E:tT}
    \begin{array}{c}
      \tT_0(x)=\frac{1}{2} \\
      \tT_1(x)=2x\tT_0(x) \\
      \tT_2(x)=2x\tT_1(x)-\tT_0(x)-\frac{1}{2} \\
      \tT_n(x)=2x\tT_{n-1}(x)-\tT_{n-2}(x)\ \text{for } \ n\geq3
    \end{array}
\end{equation}

to unify the above equalities we consider the sequences
$(f_n)_{n\in\N}$, $(g_n)_{n\in\N}$ given by
$\displaystyle{f_0=\frac{1}{4}}$,
$\displaystyle{f_2=-\frac{1}{4}}$ and $f_n=0$ otherwise,
$\displaystyle{g_0=\frac{1}{2}}$, $\displaystyle{g_2=\frac{1}{2}}$
and $g_n=0$ otherwise. We note that the equalities in
$(\ref{E:tT})$ can be converted to
\begin{equation}\label{E:tTrecurrence}
\begin{array}{c}
  \tT_0(x)=\frac{f_0}{g_0} \\
  \tT_n(x)=\left(\frac{x-g_1}{g_0}\right)\tT_{n-1}(x)-\frac{g_2}{g_0}\tT_{n-2}(x)\cdots-\frac{g_{n}}{g_0}\tT_{0}(x)+\frac{f_{n}}{g_0},
  \ \text{for} \ n\geq1
\end{array}
\end{equation}

Let $\tT=(\tit_{n,k})$ be the matrix given by
$\displaystyle{\tT_n(x)=\sum_{k=0}^n\tit_{n,k}x^k}$. 
One can verifies that
$(\ref{E:tTrecurrence})$ converts to $\tit_{n,k}=0$ if $k>n$ and
the following rules for $n\geq k$:

\[
\tit_{n,j}=-\frac{g_1}{g_0}\tit_{n-1,j}-\frac{g_2}{g_0}\tit_{n-2,j}\cdots-\frac{g_{n}}{g_0}\tit_{0,j}+
\frac{\tit_{n-1,j-1}}{g_0}\ \text{if} \ j\geq1
\]
and if $j=0$
\[
\tit_{n,0}=-\frac{g_1}{g_0}\tit_{n-1,0}-\frac{g_2}{g_0}\tit_{n-2,0}\cdots-\frac{g_{n}}{g_0}\tit_{0,0}+\frac{f_{n}}{g_0}
\]

Note that $\displaystyle{\tit_{0,0}=\frac{f_0}{g_0}}$ because the
empty sum evaluates to zero.

Using our algorithm in \cite{teo}, we obtain that $\tT$ is a
Riordan matrix. In fact we get
$\tT=\displaystyle{T\left(\frac{1}{4}-\frac{1}{4}x^2\Big|\frac{1}{2}+\frac{1}{2}x^2\right)}$
in our notation, because
$\displaystyle{f(x)=\frac{1}{4}-\frac{1}{4}x^2}$ is the generating
function of the sequence $(f_n)$ and
$\displaystyle{g(x)=\frac{1}{2}+\frac{1}{2}x^2}$ is the generating
function of the sequence $(g_n)$. So

\[
\left(
  \begin{array}{c|cccccc}
\frac{1}{4}&&&&&&\\
0&\frac{1}{2}&&&&&\\
\frac{1}{4}&0&1&&&&\\
0&-1&0&2&&&\\
0&0&-3&0&4&&\\
0&1&0&-8&0&8&\\
\vdots&\vdots&\vdots&\vdots&\vdots&\vdots&\ddots\\
  \end{array}
\right)
\]

But now more can be said because
\[
\sum_{n\geq0}\tT_n(t)x^n=T\left(\frac{1}{4}-\frac{1}{4}x^2\Big|\frac{1}{2}+\frac{1}{2}x^2\right)\left(\frac{1}{1-tx}\right)=
\frac{1}{2}\frac{1-x^2}{1+x^2-2tx}
\]
Since
\[
\sum_{n\geq0}T_n(t)x^n=\frac{1}{2}+\sum_{n\geq0}\tT_n(t)x^n
\]
we get the generating function
\[
\sum_{n\geq0}T_n(t)x^n=\frac{1-tx}{1+x^2-2tx}
\]
of the classical Chebyshev polynomials of the first kind.

Using the involved Riordan matrices we can find the known relation
between $T_n(x)$ and $U_n(x)$. Since
\[
T\left(\frac{1}{4}-\frac{1}{4}x^2\Big|\frac{1}{2}+\frac{1}{2}x^2\right)=
T\left(\frac{1}{2}-\frac{1}{2}x^2\Big|1\right)T\left(\frac{1}{2}\Big|\frac{1}{2}+\frac{1}{2}x^2\right)
\]
So, symbolically
\[
\left(
  \begin{array}{c}
\tT_0(x)\\
\tT_1(x)\\
\tT_2(x)\\
\tT_3(x)\\
\tT_4(x)\\
\tT_5(x)\\
\vdots\\
  \end{array}
\right)= \left(
  \begin{array}{ccccccc}
\frac{1}{2}&&&&&&\\
0&\frac{1}{2}&&&&&\\
-\frac{1}{2}&0&\frac{1}{2}&&&&\\
0&-\frac{1}{2}&0&\frac{1}{2}&&&\\
0&0&-\frac{1}{2}&0&\frac{1}{2}&&\\
0&0&0&-\frac{1}{2}&0&\frac{1}{2}&\\
\vdots&\vdots&\vdots&\vdots&\vdots&\vdots&\ddots\\
  \end{array}
\right) \left(
  \begin{array}{c}
U_0(x)\\
U_1(x)\\
U_2(x)\\
U_3(x)\\
U_4(x)\\
U_5(x)\\
\vdots\\
  \end{array}
\right)
\]

and consequently

\[\tT_n(x)=-\frac{1}{2}U_{n-2}(x)+\frac{1}{2}U_n(x)\]
 or
\[
2\tT_n(x)=U_n(x)-U_{n-2}(x)
\]
and then
\[
2T_n(x)=U_n(x)-U_{n-2}(x),\  n\geq 3
\]
\end{ejem}

As we noted in Section 4 of \cite{2ways}, if we delete the first
row and the first column in the Riordan matrix $T(f\mid g)$ we
obtain the new Riordan matrix
$\displaystyle{T\left(\frac{f}{g}\Big| g \right)}$. On the other
hand to add suitably a new column to the left of $T(f\mid g)$ one
place shifted up, and complete the new first row only with zeros
we have the Riordan matrix $T(fg\mid g )$. So deleting or adding
in the above sense any amount of rows and columns to $T(f\mid g)$
we obtain the intrisically related family of Riordan matrices
\[\cdots,T(g^3f\mid g),T(g^2f\mid g),T(gf\mid g),\mathbf{T(f\mid g)}, T(\frac{f}{g}\mid g), T(\frac{f}{g^2}\mid g),T(\frac{f}{g^3}\mid g),\cdots\]
We can easily obtain a recurrence to get the associated
polynomials to $\displaystyle{T\left(\frac{f}{g^n}\Big| g
\right)}$ in terms of that of $T(f\mid g )$. We have an analogous
conclusion on $T(fg^n\mid g )$ $n\geq0$. Anyway, once we know the
polynomial associated to $T(f\mid g )$ we can calculate that of
$T(fg^n\mid g )$ for $n\in\Z$.

\begin{prop}\label{P:recuOrbita}
Let  $\displaystyle{f=\sum_{n\geq0}f_nx^n}$,
$\displaystyle{g=\sum_{n\geq0}g_nx^n}$ be two power series such
that $f_0\neq0, \ g_0\neq0$. Suppose that $(p_n(x))_{n\in\N}$ is
the associated polynomial sequence of the Riordan array $T(f\mid g
)$, then

(a) If $(q_n(x))_{n\in\N}$ is the associated sequence to $T(fg\mid
g )$ we obtain
\[
q_n(x)=xp_{n-1}(x)+f_n \ \text{ if } \ n\geq1
\]
and $q_0(x)=f_0$.

(b) If $(r_n(x))_{n\in\N}$ is the associated polynomial sequence
to $\displaystyle{T\left(\frac{f}{g}\Big| g \right)}$ then
\[
r_{n-1}(x)=\frac{p_n(x)-p_n(0)}{x}\ \text { for } \ n\geq1
\]
\end{prop}
\begin{proof}
(a) $T(fg\mid g )=T(g\mid 1)T(f\mid g )$. Using the umbral
composition we have
\[
q_n(x)=g_np_0(x)+g_{n-1}p_1(x)+\cdots+g_0p_n(x)
\]
Using now our Theorem \ref{T:recu} we obtain
\[
q_n(x)=g_np_0(x)+g_{n-1}p_1(x)+\cdots+g_0\left(\left(\frac{x-g_1}{g_0}\right)p_{n-1}(x)-\frac{g_2}{g_0}p_{n-2}(x)\cdots-\frac{g_{n}}{g_0}p_{0}(x)+\frac{f_{n}}{g_0}\right)
\]
consequently
\[
q_n(x)=xp_{n-1}(x)+f_n
\]

(b) Now $\displaystyle{T(g\mid 1)T\left(\frac{f}{g}\Big| g
\right)}=T(f\mid g )$. So
\[
p_n(x)=g_nr_0(x)+g_{n-1}r_1(x)+\cdots+g_0r_n(x)
\]
using again the Theorem \ref{T:recu} for the sequences $r_n(x)$ we
obtain
\[
p_n(x)=g_nr_0(x)+g_{n-1}r_1(x)+\cdots+g_0\left(\left(\frac{x-g_1}{g_0}\right)r_{n-1}(x)-\frac{g_2}{g_0}r_{n-2}(x)
\cdots-\frac{g_{n}}{g_0}r_{0}(x)+\frac{d_{n}}{g_0}\right)
\]
where the $d_n$ is the $n$-coefficient of the series
$\displaystyle{\frac{f}{g}}$. Consequently
$p_n(x)=xr_{n-1}(x)+d_n$. Note that $p_n(0)=d_n$, so
\[
r_{n-1}(x)=\frac{p_n(x)-p_n(0)}{x} \ \text{ if } \ n\geq1
\]
\end{proof}

\begin{cor}
Suppose $\displaystyle{g=\sum_{n\geq0}g_nx^n}$ with $g_0\neq0$.
Let $(p_n(x))_{n\in\N}$ be the polynomial sequence associated to
$T(1\mid g )$ and $(q_n(x))_{n\in\N}$  that associated to $T(g\mid
g )$. Then:
\[
q_n(x)=xp_{n-1}(x)\  \text{ for } \ n\geq1 \ \text{ and } \
q_0(x)=1
\]
\end{cor}

\begin{ejem}
As an application of Proposition \ref{P:recuOrbita} and as we
noted in Section \ref{S:ejem}, the relationships between both kind
of Morgan-Voyce polynomials are
\[
    B_n(x)-B_{n-1}(x)=b_{n}(x)
\]
\[
 b_n(x)-b_{n-1}(x)=xB_{n-1}(x)
\]
That in terms of Riordan arrays this means
\[
T(1-x|1)T(1|(1-x)^2)=T(1-x|(1-x)^2)
\]
\[
T(1-x|1)T(1-x|(1-x)^2)=T((1-x)^2|(1-x)^2)
\]
because $(T(1\mid (1-x)^2))$ gives rise to $(B_n(x))$ and
$T(1-x|(1-x)^2)$ gives rise to $(b_n(x))$
\end{ejem}

In the following expressions we consider $(p_n(x))$ as the family
of polynomials associated to $T(f\mid g)$, and we denote
$(q_n(x))$ the family of polynomials associated  to each of the
products of matrices, moreover $a, b$ are constant series with
$b\neq0$:

\[
T(a\mid1)T(f\mid g)=T(af\mid g), \ \text{then} \ q_n(x)=ap_n(x)
\]
\[
T(1\mid b)T(f\mid g)=T\left(f\left(\frac{x}{b}\right)\mid
bg\left(\frac{x}{b}\right)\right), \ \text{then} \
q_n(x)=\frac{1}{b^{n+1}}p_n(x)
\]
\[
T(f\mid g)T(a\mid 1)=T(af\mid g), \ \text{then} \ q_n(x)=ap_n(x)
\]
\[
T(f\mid g)T(1\mid b)=T(f\mid bg), \ \text{then} \
q_n(x)=\frac{1}{b}p_n\left(\frac{x}{b}\right)
\]

The above results can be summarized and extended in the following
way:

\begin{prop}\label{P:const}
Let $T(f\mid g)$ and $T(l\mid m)$ be two element of the Riordan
group. Suppose that $(p_n(x))$ and $(q_n(x))$ are the corresponding
associated families of polynomials. Suppose also that
\[
T(l\mid m)=T(\gamma\mid\alpha+\beta x)T(f\mid g)T(c\mid a+bx)
\]
where $\alpha, \gamma, a, c\neq0$. Then
\[
q_n(x)=\frac{\gamma c}{\alpha
a}\left(\sum_{k=0}^n\binom{n}{k}\left(-\frac{\beta}{\alpha}\right)^{n-k}\frac{1}{\alpha^k}p_k\left(\frac{x-b}{a}\right)\right)
\]
\end{prop}
\begin{proof}
Using Theorem \ref{T:recu} we have that if $(s_n(x))$ is the
family of polynomials associated to $T(\gamma\mid\alpha+\beta x)$
then
\[
s_0(x)=\frac{\gamma}{\alpha} \ \qquad \ \text{and} \ \qquad
s_n(x)=\left(\frac{x-\beta}{\alpha}\right)s_{n-1(x)} \ \forall n
\geq1
\]
consequently
\[
s_n(x)=\frac{\gamma}{\alpha}\left(\frac{x-\beta}{\alpha}\right)^n
\qquad n\in\N
\]

Proposition 14 in \cite{BanPas} says that if $(r_n(x))$ is the
family of polynomials associated to $T(f\mid g)T(c\mid a+bx)$ then
\[
r_n(x)=\frac{c}{a}p_n\left(\frac{x-b}{a}\right)
\]
Since $(q_n(x))=(s_n(x))\sharp(r_n(x))$ we obtain that
\[
(q_n(x))=\left(\frac{\gamma}{\alpha}\sum_{k=0}^n\binom{n}{k}
\left(-\frac{\beta}{\alpha}\right)^{n-k}\frac{1}{\alpha^k}x^k\right)\sharp
\left(\frac{c}{a}p_n\left(\frac{x-b}{a}\right)\right)
\]
Hence
\[
q_n(x)=\frac{\gamma c}{\alpha
a}\left(\sum_{k=0}^n\binom{n}{k}\left(-\frac{\beta}{\alpha}\right)^{n-k}\frac{1}{\alpha^k}p_k\left(\frac{x-b}{a}\right)\right)
\]
\end{proof}

\begin{ejem}
As we noted in Section \ref{S:ejem} the relation between the Pell
and the Fibonacci polynomials is $P_n(x)=F_n(2x)$. Recall
\[
T\left(\frac{1}{2}\Big|1\right)T(1|1-x^2)T\left(1\Big|\frac{1}{2}\right)=T\left(\frac{1}{2}\Big|\frac{1}{2}-\frac{1}{2}x^2\right)
\]
and
$\displaystyle{T\left(\frac{1}{2}\Big|\frac{1}{2}-\frac{1}{2}x^2\right)}$
gives rise to the Pell polynomials and $T(1\mid 1-x^2)$ gives rise
to the Fibonacci polynomials.
\end{ejem}

\begin{ejem} Recall that the Fermat polynomials are the polynomials given by $\mathcal{F}_0(x)=1$,
$\mathcal{F}_1(x)=3x$ and
\[
\mathcal{F}_{n}(x)=3x\mathcal{F}_{n-1}-2\mathcal{F}_{n-2} \ \text{
for } \ n\geq2
\]
Using our Theorem \ref{T:recu}, this means that Fermat polynomials
are the polynomials associated to the Riordan matrix
$\displaystyle{T\left(\frac{1}{3}\Big|\frac{1}{3}+\frac{2}{3}x^2\right)}$.
For this case, $g_0=\frac{1}{3}, \ g_1=0, \ g_2=\frac{2}{3}, \
g_n=0, \ \forall n\geq3$ and $f_0=\frac{1}{3}, \ f_n=0 \ \forall
n\geq1$. And the rule of construction of this triangle is:
$d_{n,k}=-2d_{n-2,k}+3d_{n-1,k-1}$ for $k>0$. The few first rows
are:
\[
\left(
  \begin{array}{c|cccccccc}
\frac{1}{3}&&&&&&\\
0&1&&&&&\\
0&0&3&&&&\\
0&-2&0&9&&&\\
0&0&-12&0&27&&\\
0&4&0&-54&0&81&\\
0&0&36&0&-216&0&243\\
0&-8&0&216&0&-810&0&729\\
\vdots&\vdots&\vdots&\vdots&\vdots&\vdots&\vdots&\vdots&\ddots\\
  \end{array}
\right)
\]
Consequently the few first Fermat polynomials are

\noindent $\mathcal{F}_0(x)=1$\\
\noindent $\mathcal{F}_1(x)=3x$\\
\noindent $\mathcal{F}_2(x)=-2+9x^2$\\
\noindent $\mathcal{F}_3(x)=-12x+27x^3$\\
\noindent $\mathcal{F}_4(x)=4-54x^2+81x^4$\\
\noindent $\mathcal{F}_5(x)=36x-216x^3+243x^5$\\
\noindent $\mathcal{F}_6(x)=-8+216x^2-810x^4+729x^6$\\


Since
\[
T\left(\frac{1}{3}\Big|\frac{1}{3}(1+2x^2)\right)=T\left(1\Big|\frac{1}{\sqrt{2}}\right)T\left(\frac{1}{2}\Big|\frac{1}{2}(1+x^2)\right)T\left(\frac{2}{3}\Big|\frac{2\sqrt{2}}{3}\right)
\]
and using Proposition \ref{P:const} we obtain the following relation
to the Chebysev polynomials of the second kind:
\[
\mathcal{F}_n(x)=(\sqrt{2})^{n}U_n\left(\frac{3x}{2\sqrt{2}}\right)
\]
Recently, it has been introduced by Boubaker et al. a special
family of polynomials in \cite{Boubaker}, \cite{otrosBoubaker}
related to the so called spray pyrolysis techniques. Now we are
going to find a relation of these polynomials with the Chebysev
polynomials of the second kind and then also with the Fermat
polynomials as showed above. This new sequences of  polynomials is
given by $\mathcal{B}_0(x)=1$, $\mathcal{B}_1(x)=x$,
$\mathcal{B}_2(x)=2+x^2$ and
\[
\mathcal{B}_{n}(x)=x\mathcal{B}_{n-1}(x)-\mathcal{B}_{n-2}(x) \
\text{ for } \ n\geq3
\]
Using our Theorem \ref{T:recu}, this means that
$\mathcal{B}_{n}(x)$ polynomials are the polynomials associated to
the Riordan matrix $\displaystyle{T\left(1+3x^2\mid
1+x^2\right)}$. For this case, $g_0=1, \ g_1=0, \ g_2=1, \ g_n=0,
\ \forall n\geq3$ and $f_0=1, f_1=0, f_2=3 \ f_n=0 \ \forall
n\geq3$. And the rule of construction of this triangle is:
$d_{n,k}=-d_{n-2,k}+d_{n-1,k-1}$, then
\[
\left(
  \begin{array}{c|cccccccc}
1&&&&&&\\
0&1&&&&&\\
3&0&1&&&&\\
0&2&0&1&&&\\
0&0&1&0&1&&\\
0&-2&0&0&0&1&\\
0&0&-3&0&-1&0&1\\
0&2&0&-3&0&-2&0&1\\
\vdots&\vdots&\vdots&\vdots&\vdots&\vdots&\vdots&\vdots&\ddots\\
  \end{array}
\right)
\]
Consequently the few first associated polynomials are

\noindent $\mathcal{B}_0(x)=1$\\
\noindent $\mathcal{B}_1(x)=x$\\
\noindent $\mathcal{B}_2(x)=2+x^2$\\
\noindent $\mathcal{B}_3(x)=x+x^3$\\
\noindent $\mathcal{B}_4(x)=-2+x^4$\\
\noindent $\mathcal{B}_5(x)=-3x-x^3+x^5$\\
\noindent $\mathcal{B}_6(x)=2-3x^2-2x^4+x^6$\\


with generating function
\[
\sum_{n\geq0}\mathcal{B}_n(t)x^n=T\left(1+3x^2\mid
1+x^2\right)\left(\frac{1}{1-xt}\right)=\frac{1+3x^2}{1-xt+x^2}
\]
 Since
\[
T\left(1+3x^2\mid 1+x^2\right)=T\left(1+3x^2\mid 1\right)
T\left(\frac{1}{2}\Big|\frac{1}{2}(1+x^2)\right)T(2\mid 2)
\]
and using Proposition \ref{P:relat1pn} and Proposition \ref{P:const}
we obtain the following relation to the Chebysev polynomials of the
second kind:
\[
\mathcal{B}_n(x)=U_{n}\left(\frac{x}{2}\right)+3U_{n-2}\left(\frac{x}{2}\right)\
\text{ for } \ n\geq2
\]
\end{ejem}

\section{Some applications to the generalized umbral calculus:
the associated polynomials and its recurrence relations.}

There are many other types of polynomial sequences in the literature
that can be constructed by means of Riordan arrays. We are going to
characterize by means of recurrences relations all the polynomial
sequences called \emph{generalized Appell polynomials} in Boas-Buck
\cite{Boas-Buck} page 17-18. We will follow their definitions there.

We first introduce some concepts. Suppose we have any polynomial
sequence $(p_n(x))_{n\in\N}$ with
$\displaystyle{p_n(x)=\sum_{k=0}^np_{n,k}x^k}$ and let
$h(x)=\sum_{n\geq0}h_{n}x^n$ any power series, we call the
\textit{Hadamard $h$-weighted sequence} generated by $(p_n(x))$ to
the sequence $\displaystyle{p_n^h(x)=(p_n\star h)(x)}$ where $\star$
means the Hadamard product of series. Recall that if
$f=\sum_{n\geq0}f_{n}x^n$ and $g=\sum_{n\geq0}g_{n}x^n$, then the
Hadamard product $f\star g$ is given by $f\star
g=\sum_{n\geq0}f_{n}g_nx^n$.

Note that $p_n^h$ is a polynomial for every $n\in\N$ and
$h\in\K[[x]]$. In fact $p_n^h(x)=\sum_{k=0}^np_{n,k}h_kx^k$.

Note also that the original definition of generalized Appell
polynomials defined by Boas-Buck in \cite{Boas-Buck} can be rewriten
in terms of Riordan matrices in the following way

\begin{prop} A sequence of polynomials $(s_n(x))$ is a family
of generalized Appell polynomials if and only if there are three
series $f, g, h\in \K[[x]]$, $f=\sum_{n\geq0}f_{n}x^n$,
$g=\sum_{n\geq0}g_{n}x^n$ and $h(x)=\sum_{n\geq0}h_{n}x^n$   with
$f_0,g_0\neq0$, and $h_n\neq0$ for all $n$ such that
\[
T(f\mid g)h(tx)=\sum_{n\geq0}s_n(t)x^n
\]
Moreover in this case, $\displaystyle{s_n(x)=p_n^h(x)}$ in the
above sense where $(p_n(x))$ is the associated polynomial sequence
of $T(f\mid g)$. Consequently
\[
\sum_{n\geq0}s_n(t)x^n=\sum_{n\geq0}(p_n\star h)(t)x^n=
\frac{f(x)}{g(x)}h\left(t\frac{x}{g(x)}\right)
\]
\end{prop}
\begin{proof}
If $\displaystyle{T(f\mid g)(h(tx))=\sum_{n\geq0}s_n(t)x^n}$ then
obviously $(s_n(x))$ is a generalized Appell sequence because
\[
\sum_{n\geq0}s_n(t)x^n=\frac{f(x)}{g(x)}h\left(t\frac{x}{g(x)}\right)
\]
Suppose now that $(s_n(x))$ is a generalized Appell sequence, then
there are three series $A, B, \Phi$ where
$\displaystyle{A=\sum_{n\geq0}A_nx^n}$, $A_0\neq0$,
$\displaystyle{B=\sum_{n\geq1}B_nx^n}$, $B_1\neq0$ and
$\displaystyle{\Phi=\sum_{n\geq0}\Phi_nx^n}$ with $\Phi_n\neq0, \
\forall n\in \N$ such that
\[
\sum_{n\geq0}s_n(t)x^n=A(x)\Phi(tB(x))
\]
If we take $\Phi=h$, $\displaystyle{g(x)=\frac{x}{B(x)}}$ and
$\displaystyle{f(x)=\frac{xA(x)}{B(x)}}$ we are done.
\end{proof}

\begin{rmk}
Note that if $\displaystyle{h(x)=\frac{1}{1-x}}$ the family of
$(p_n^{\frac{1}{1-x}}(x))$ is exactly the associated polynomials
$(p_n(x))$ of $T(f\mid g)$, because $\displaystyle{\frac{1}{1-x}}$
is the neutral element in the Hadamard product.
\end{rmk}

\begin{ejem} \textbf{The Sheffer polynomials. } Following the previous
proposition we have that $(S_n(x))$ is a Sheffer sequence if and
only if there is a Riordan matrix $T(f\mid g)$ such that
\[
T(f\mid g)(e^{tx})=\sum_{n\geq0}S_n(t)x^n
\]
The usual way to introduce Sheffer sequences is by means of the
corresponding generating function
\[ \sum_{n\geq0}S_{n}(t)x^n=A(x)e^{tH(x)}\]
where $\displaystyle{A=\sum_{n\geq0}A_nx^n, H=\sum_{n\geq1}H_nx^n}$
with $A_0\neq0$, $H_1\neq0$. Note that for this case the
corresponding Riordan matrix is
\[
T\left(\frac{xA(x)}{H(x)}\Big| \frac{x}{H(x)}\right)
\]
The general term of a Sheffer sequence, $S_n(x)$ is given by
\[
S_n(x)=p_n(x)\star e^x
\]
where $(p_n(x))$ are the associated polynomials to $T(f\mid g)$.
Consequently
\[
\ S_n(x)=\sum_{k=0}^n \frac{p_{n,k}}{k!}x^k
\]
if $\displaystyle{p_n(x)=\sum_{k=0}^np_{n,k}x^k}$.

 \textbf{WARNING} Note that in many places \cite{Roman},
\cite{Roman-Rota}, \cite{Rota2} they call a Sheffer sequence to
the sequence $(n!S_n(x))_{n\in\N}$ where $(S_n(x))_{n\in\N}$ is
our Sheffer sequence.
\end{ejem}

In the following example we can note that applying a fixed
$T(f\mid 1)$ to different series $h$ gives rise to some classical
families of polynomials.

\begin{ejem}\textbf{The Brenke polynomials}
Following \cite{Boas-Buck}, $(B_n(x))$ is in  the class of Brenke
polynomials if
\[
T(f\mid 1)(h(tx))=\sum_{n\geq0}B_n(t)x^n
\]
Some particular cases are:
\[
T(f\mid
1)\left(\frac{1}{1-tx}\right)=\sum_{n\geq0}T^{\ast}_n(t)x^n
\]
where ($T^{\ast}_n)$ are the reversed Taylor polynomial of $f$.
\[
T(f\mid 1)(e^{tx})=\sum_{n\geq0}A_n(t)x^n
\]
where $(A_n(x))$ are the Appell polynomials of $f$.
\end{ejem}

Using analogous arguments as in the previous section for
polynomials of Riordan type, we can get some relationships between
some classical Sheffer sequences once we know, easily, some
relation between their corresponding Riordan matrices.

\begin{ejem}
\textbf{Pidduck and Mittag-Leffler polynomials. } Consider the
sequence $(P_n(x))$ satisfying
\[
\sum_{n\geq0}\mathcal{P}_n(t)x^n=T\left(\frac{x}{(1-x)\log\left(\frac{1+x}{1-x}
\right)}\Big|\frac{x}{\log\left(\frac{1+x}{1-x}
\right)}\right)(e^{tx})
\]
in matricial form:
\[
\left(%
\begin{array}{cccccc}
  1 & 0 & 0 & 0 & 0 & \cdots\\
  1 & 2 & 0 & 0 & 0 & \cdots\\
  1 & 2 & 4 & 0 & 0 & \cdots\\
  1 & \frac{8}{3} & 4 & 8 & 0 & \cdots\\
  1 & \frac{8}{3} & \frac{20}{3} & 8 & 16 & \cdots\\
  \vdots & \vdots & \vdots & \vdots & \vdots & \ddots\\
\end{array}%
\right) \left(
 \begin{array}{c}
 1\\
 t\\
 \frac{t^2}{2}\\
 \frac{t^3}{6}\\
 \frac{t^4}{24}\\
 \vdots
  \end{array}
 \right)=
 \left(
 \begin{array}{c}
 1\\
 2t+1\\
 2t^2+2t+1\\
 \frac{4}{3}t^3+2t^2+\frac{8}{3}t+1\\
 \frac{2}{3}t^4+\frac{4}{3}t^3+\frac{10}{3}t^2+\frac{8}{3}t+1\\
 \vdots
  \end{array}
 \right)
\]
If we take $\widetilde{P}_n(x)=n!\mathcal{P}_n(x)$, then
$\widetilde{P}_n(x)$ are the usual Pidduck polynomials.
\[
\begin{array}{c}
  \widetilde{P}_0(x)=1 \\
  \widetilde{P}_1(x)=2x+1 \\
  \widetilde{P}_2(x)=4x^2+4x+2 \\
  \widetilde{P}_3(x)=8x^3+12x^2+16x+6 \\
  \widetilde{P}_4(x)=16t^4+32x^3+80x^2+64x+24
\end{array}
\]
On the other hand we get the Mittag-Leffler polynomials, in the
following way. If $(M_n(x))$ is given by the formula:
\[
\sum_{n\geq0}M_n(t)x^n=T\left(\frac{x}{\log\left(\frac{1+x}{1-x}
\right)}\Big|\frac{x}{\log\left(\frac{1+x}{1-x}
\right)}\right)(e^{tx})
\]
in matricial form:
\[
\left(%
\begin{array}{cccccc}
  1 & 0 & 0 & 0 & 0 & \cdots\\
  0 & 2 & 0 & 0 & 0 & \cdots\\
  0 & 0 & 4 & 0 & 0 & \cdots\\
  0 & \frac{2}{3} & 0 & 8 & 0 & \cdots\\
  0 & 0 & \frac{8}{3} & 0 & 16 & \cdots\\
  \vdots & \vdots & \vdots & \vdots & \vdots & \ddots\\
\end{array}%
\right) \left(
 \begin{array}{c}
 1\\
 t\\
 \frac{t^2}{2}\\
 \frac{t^3}{6}\\
 \frac{t^4}{24}\\
 \vdots
  \end{array}
 \right)=
 \left(
 \begin{array}{c}
 1\\
 2t\\
 2t^2\\
 \frac{4}{3}t^3+\frac{2}{3}t\\
 \frac{2}{3}t^4+\frac{4}{3}t^2\\
 \vdots
  \end{array}
 \right)
\]
then, if we take now $\widetilde{M}_n(x)=n!M_n(x)$, then
$\widetilde{M}_n(x)$ are the usual Mittag-Leffler polynomials.
\[
\begin{array}{c}
  \widetilde{M}_0(x)=1 \\
  \widetilde{M}_1(x)=2x \\
  \widetilde{M}_2(x)=4x^2 \\
  \widetilde{M}_3(x)=8x^3+4x \\
  \widetilde{M}_4(x)=16t^4+32x^2
\end{array}
\]

Both families of polynomials are related because:
\[
T\left(\frac{x}{(1-x)\log\left(\frac{1+x}{1-x}
\right)}\Big|\frac{x}{\log\left(\frac{1+x}{1-x}
\right)}\right)=T\left(\frac{1}{1-x}\Big|1\right)T\left(\frac{x}{\log\left(\frac{1+x}{1-x}
\right)}\Big|\frac{x}{\log\left(\frac{1+x}{1-x} \right)}\right)
\]
Hence
\[
\left(%
\begin{array}{c}
  \mathcal{P}_0(x) \\
  \mathcal{P}_1(x) \\
  \mathcal{P}_2(x) \\
  \mathcal{P}_3(x) \\
  \mathcal{P}_4(x) \\
  \vdots \\
\end{array}%
\right)=\left(%
\begin{array}{cccccc}
  1 & 0 & 0 & 0 & 0 & \cdots \\
  1 & 1 & 0 & 0 & 0 & \cdots \\
  1 & 1 & 1 & 0 & 0 & \cdots \\
  1 & 1 & 1 & 1 & 0 & \cdots \\
  1 & 1 & 1 & 1 & 1 & \cdots \\
  \vdots & \vdots & \vdots & \vdots & \vdots & \ddots \\
\end{array}%
\right)\left(%
\begin{array}{c}
  M_0(x) \\
  M_1(x) \\
  M_2(x) \\
  M_3(x) \\
  M_4(x) \\
  \vdots \\
\end{array}%
\right)
\]
So
\[
\mathcal{P}_n(x)=\sum_{k=0}^nM_k(x)\qquad  \text{or equivalently}
\qquad
\widetilde{P}_n(x)=\sum_{k=0}^n\binom{n}{k}(n-k)!\widetilde{M}_k(x)
\]
\end{ejem}

Using our main theorem in Section \ref{S:main} we can obtain the
following recurrence relations for the generalized Appell
polynomials, which is the main result in this section.

\begin{thm}\label{T:Tfgh}
Let $(s_n(x))_{n\in\N}$ be a sequence of polynomials with
$\displaystyle{s_n(x)=\sum_{k=0}^ns_{n,k}x^k}$. Then
$(s_n(x))_{n\in\N}$ is a family of generalized Appell polynomials
if and only if there are three sequences $(f_n)$, $(g_n)$,
$(h_n)\in \K$ with $f_0,g_0\neq0$ and $h_n\neq0 \ \forall n\in\N$
such that
\[
s_n(x)=\frac{1}{g_0}(xs_{n-1}(x)\star
\widehat{h}(x))-\frac{g_1}{g_0}s_{n-1}(x)-\cdots-\frac{g_n}{g_0}s_0(x)+\frac{h_0f_n}{g_0}\
\quad \forall n\in \N \ \quad \text{with} \ \quad
s_0(x)=\frac{h_0f_0}{g_0}
\]
where
$\displaystyle{\widehat{h}(x)=\sum_{k=1}^{\infty}\frac{h_k}{h_{k-1}}x^k}$.
Moreover the coefficients of this family of polynomials satisfy
the following recurrence: \noindent If $k\geq1$
\[
s_{n,k}=-\frac{g_1}{g_0}s_{n-1,k}-\cdots-\frac{g_n}{g_0}s_{0,k}+\frac{h_k}{h_{k-1}}s_{n-1,k-1}
\]
If $k=0$
\[
s_{n,0}=-\frac{g_1}{g_0}s_{n-1,0}-\cdots-\frac{g_n}{g_0}s_{0,0}+\frac{h_0f_n}{g_0},\qquad
s_{0,0}=\frac{h_0f_0}{g_0}
\]
\end{thm}
\begin{proof}
If $(s_n(x))$ is a family of generalized Appell polynomials then
there are three sequence $(f_n)$, $(g_n)$, $(h_n)$ of elements in
$\K$ with $f_0,g_0\neq0$ and $h_n\neq0 \ \forall n\in\N$, such
that if $\displaystyle{f=\sum_{n\geq0}f_nx^n,
g=\sum_{n\geq0}g_nx^n}$ and $\displaystyle{h=\sum_{n\geq0}h_nx^n}$
then
\[
T(f\mid g)h(tx)=\sum_{n\geq0}s_n(t)x^n
\]
since  $\displaystyle{s_n(x)=p_n^h(x)=p_n(x)\star h(x)}$, the
family of polynomials $(p_n(x))$ associated to $T(f\mid g)$ obeys
the recurrence relation of Theorem \ref{T:recu}: Using the
distributivity of Hadamard product we get
\[
p_n(x)\star h(x)=\left(\frac{x-g_1}{g_0}\right)p_{n-1}(x)\star
h(x)- \frac{g_2}{g_0}p_{n-2}(x)\star
h(x)\cdots-\frac{g_{n}}{g_0}p_{0}(x)\star
h(x)+\frac{f_{n}}{g_0}\star h(x)=
\]
\[
p_n^h(x)=\frac{x}{g_0}p_{n-1}(x)\star
h(x)-\frac{g_1}{g_0}p_{n-1}^h(x)-
\frac{g_2}{g_0}p_{n-2}^h(x)\cdots-\frac{g_{n}}{g_0}p_{0}^h(x)+\frac{f_{n}h_0}{g_0}
\]
since
\[
xp_{n-1}(x)\star
h(x)=p_{n-1,0}h_1x+p_{n-1,1}h_2x^2+\cdots+p_{n-1,n-1}h_{n}x^{n}
\]
then
\[
xp_{n-1}(x)\star
h(x)=p_{n-1,0}h_0\frac{h_1}{h_0}x+p_{n-1,1}h_1\frac{h_2}{h_1}x^2+
\cdots+p_{n-1,n-1}h_{n-1}\frac{h_n}{h_{n-1}}x^{n}=xp^h_{n-1}(x)\star
\widehat{h}(x)
\]
so we get the result.

On the other hand if there are three sequences $(f_n)$, $(g_n)$,
$(h_n)\in \K$ with $f_0,g_0\neq0$ and $h_n\neq0 \ \forall n\in\N$
such that
\[
s_n(x)=\frac{1}{g_0}(xs_{n-1}(x)\star
\widehat{h}(x))-\frac{g_1}{g_0}s_{n-1}(x)-\cdots-\frac{g_n}{g_0}s_0(x)+\frac{h_0f_n}{g_0}\
\quad \forall n\in \N \ \quad \text{with} \ \quad
s_0(x)=\frac{h_0f_0}{g_0}
\]
where
$\displaystyle{\widehat{h}(x)=\sum_{k=1}^{\infty}\frac{h_k}{h_{k-1}}x^k}$.
Let
\[
p_n(x)=s_n(x)\star h^{(-1)^{\star}}(x)
\]
where
$\displaystyle{h^{(-1)^{\star}}(x)=\sum_{n\geq0}\frac{1}{h_n}x^n}$.
Then
\[
s_n(x)\star h^{(-1)^{\star}}(x)=\frac{1}{g_0}(xs_{n-1}(x)\star
\widehat{h}(x))\star
h^{(-1)^{\star}}(x)-\frac{g_1}{g_0}s_{n-1}(x)\star
h^{(-1)^{\star}}(x)-\cdots-\frac{g_n}{g_0}s_0(x)\star
h^{(-1)^{\star}}(x)+\frac{h_0f_n}{g_0}\star h^{(-1)^{\star}}(x)
\]
\[
p_n(x)=\frac{1}{g_0}(xs_{n-1}(x)\star \widehat{h}(x))\star
h^{(-1)^{\star}}-\frac{g_1}{g_0}p_{n-1}(x)-\cdots-\frac{g_n}{g_0}p_0(x)
+\frac{f_n}{g_0}
\]
since
\[
xs_{n-1}(x)\star
\widehat{h}(x)=s_{n-1,0}\frac{h_1}{h_0}x+s_{n-1,1}\frac{h_2}{h_1}x^2+\cdots+
s_{n-1,n-1}\frac{h_n}{h_{n-1}}x^n
\]
then
\[
xs_{n-1}(x)\star \widehat{h}(x)\star
h^{(-1)^{\star}}(x)=xs_{n-1}(x)\star
h^{(-1)^{\star}}(x)=xp_{n-1}(x)
\]
consequently
\[
p_n(x)=\frac{1}{g_0}(xp_{n-1}(x))-\frac{g_1}{g_0}p_{n-1}(x)-\cdots-\frac{g_n}{g_0}p_0(x)
+\frac{f_n}{g_0}
\]
so $(p_n(x))$ obeys Theorem \ref{T:recu} and then $(p_n(x))$ is
the associated polynomials to $T(f\mid g)$. Hence $(s_n(x))$ is a
family of generalized Appell polynomials.

The second part of the result is an easy consequence of our
Algorithm \ref{A:algo} in the Introduction.
\end{proof}
%
%

\begin{rmk}
Note that if $k\geq1$, some terms in the recurrence are null, in
fact $s_{l,k}=0$ if $l< k$. Consequently:
\[
s_{n,k}=-\frac{g_1}{g_0}s_{n-1,k}-\cdots-\frac{g_{n-k}}{g_0}s_{k,k}+\frac{h_k}{h_{k-1}}s_{n-1,k-1}
\]
\end{rmk}

A consequence that we can obtain from the recurrence relation for
the generalized Appell sequences is the following relation between
the Hadamard $h$-weighted and $h'$-weighted sequences for a
polynomials sequence of Riordan type. For notational convenience
we represent now by $\mathcal{D}(\alpha)$ to the derivative of any
series $\alpha$. The result obtained below  when we consider the
classical Appell sequences, is just what Appell took as the
definition for these classical sequences.
\begin{cor}\label{C:DerPn}
Let $T(f\mid g)$ be any element of the Riordan group with
$\displaystyle{f=\sum_{n\geq0}f_nx^n}$,
$\displaystyle{g=\sum_{n\geq0}g_nx^n}$, and with associated
sequence $(p_n(x))$. Suppose that $h\in\K[[x]]$ is Hadamard
invertible. Then the $\mathcal{D}(h)$ is Hadamard invertible and
\[
p_{n-1}^{\mathcal{D}(h)}(x)=\sum_{k=0}^{n}g_k\mathcal{D}(p_{n-k}^h)(x)
\]
\end{cor}
\begin{proof}
We know that
\[
p_n^h(x)=\frac{1}{g_0}(xp_{n-1}^h(x)\star
\widehat{h}(x))-\frac{g_1}{g_0}p^h_{n-1}(x)-\cdots-\frac{g_n}{g_0}p^h_0(x)+\frac{h_0f_n}{g_0}
\]
Applying the derivative in both sides we obtain
\[
\mathcal{D}(p_n^h)(x)=\frac{1}{g_0}\mathcal{D}(xp_{n-1}^h(x)\star
\widehat{h}(x))-\sum_{k=1}^n\frac{g_k}{g_0}\mathcal{D}(p^h_{n-k})(x)
\]
Consequently
\[
\mathcal{D}(xp_{n-1}^h(x)\star
\widehat{h}(x))=\sum_{k=0}^ng_k\mathcal{D}(p^h_{n-k})(x)
\]
It is easy to prove that
\[
\mathcal{D}(m(x)\star
l(x))=\frac{m(x)-m(0)}{x}\star\mathcal{D}(l(x))=
\mathcal{D}(m(x))\star\frac{(l(x)-l(0)}{x}
\]
for any series $l,m\in\K[[x]]$. Using the first equality above we
get
\[
p_{n-1}^h\star\mathcal{D}(\widehat{h})(x)=\sum_{k=0}^ng_k\mathcal{D}(p_{n-k}^h)(x)
\]
but
\[
(p_{n-1}(x)\star
h(x))\star\mathcal{D}(\widehat{h})(x)=p_{n-1}(x)\star (h(x)\star
(\mathcal{D}(\widehat{h})(x))
\]
and since
$\displaystyle{\widehat{h}(x)=\sum_{k\geq1}\frac{h_k}{h_{k-1}}x^k}$
we obtain that
\[
h(x)\star\mathcal{D}(\widehat{h})(x)=\mathcal{D}(h)(x)
\]
and so we have the announced equality.
\end{proof}

The previous result convert to the following formulas in the
important class of Sheffer sequences.
\begin{ejem} \textbf{The recurrence relation for the Sheffer polynomials. }
Since $\displaystyle{h(x)=e^x=\sum_{n\geq0}\frac{x^n}{n!}}$ and
$\displaystyle{\widehat{h}(x)=\sum_{n\geq1}\frac{x^n}{n}=-\log(1-x)}$,
the recurrence relation is:
\[
S_n(x)=\frac{1}{g_0}(xS_{n-1}(x)\star
(-\log(1-x)))-\frac{g_1}{g_0}S_{n-1}(x)-\cdots-\frac{g_n}{g_0}S_0(x)+
\frac{f_n}{g_0}\ \quad \forall n\in \N \ \quad \text{with} \ \quad
S_0(x)=\frac{f_0}{g_0}
\]

and the recurrence relations for the coefficients are

\noindent If $k\geq1$
\[
S_{n,k}=-\frac{g_1}{g_0}S_{n-1,k}-\cdots-\frac{g_n}{g_0}S_{0,k}+\frac{1}{k}S_{n-1,k-1}
\]
If $k=0$
\[
S_{n,0}=-\frac{g_1}{g_0}S_{n-1,0}-\cdots-\frac{g_n}{g_0}S_{0,0}+\frac{f_n}{g_0},\qquad
S_{0,0}=\frac{f_0}{g_0}
\]

And for its derivatives. Since
\[
(xS_{n-1}(x)\star(-\log(1-x)))'=S_{n-1}(x)\star\frac{1}{1-x}=S_{n-1}(x)
\]
Then
\[
S'_n(x)=\frac{1}{g_0}S_{n-1}(x)-\frac{g_1}{g_0}S'_{n-1}(x)-\cdots-\frac{g_n}{g_0}S'_0(x)
\]
So
\[
S_{n-1}(x)=\sum_{k=0}^ng_kS'_{n-k}(x)
\]
\end{ejem}

In some cases the above formulas allow us to compute easily some
generalized Appell sequences in terms of the associated sequences
of Riordan type.


\begin{ejem} \textbf{Some easy computations related to the geometric series.}
Let $(p_n(x))$ be a polynomial sequence of Riordan type. Then
\begin{itemize}
    \item [(i)] \[
    p_n^{\frac{1}{(1-x)^2}}(x)=xp'_n(x)+p_n(x)=(xp_n(x))'\ \qquad
    \forall n\geq0
    \]
    The proof of the above equality is the following
    \[ p_n^{\frac{1}{(1-x)^2}}(x)=p_n(x)\star\frac{1}{(1-x)^2}=
p_n(x)\star\left(\frac{1}{(1-x)}\right)'=
\left(xp_n(x)\star\frac{1}{(1-x)}\right)'=(xp_n(x))'
\]
    \item [(ii)] If $a\neq0$ then
    \[
    p_n^{a-\log(1-x)}(x)=ap_n(0)+\int_0^x\frac{p_n(t)-p_n(0)}{t} \ \qquad
    \forall n\geq0
    \]
    The proof of the last equality is
    \[
p_n^{a-\log(1-x)}(x)=p_n(x)\star(a-\log(1-x))
\]
So $\displaystyle{p_n^{a-\log(1-x)}(0)=ap_n(0)}$. The derivative
in the right part of the equality is
\[
\frac{p_n(x)-p_n(0)}{x}\star\frac{1}{(1-x)}=\frac{p_n(x)-p_n(0)}{x}
\]
Consequently
\[
   p_n^{a-\log(1-x)}(x)=ap_n(0)+\int_0^x\frac{p_n(t)-p_n(0)}{t}
\]
\end{itemize}
\end{ejem}


The following examples are particular cases of Sheffer polynomials
which can be easily described with a different representation as
generalized Appell polynomial. In fact any Sheffer sequence can be
obtained as a Hadamard $h$-weighted sequences polynomials for some
$h(x)\neq e^{x}$. We choose, in particular, Laguerre sequence
because it is very close to the Pascal triangle.

\begin{ejem} \textbf{The Laguerre polynomials.} We consider
\[
T(-1\mid
x-1)(e^{tx})=T(1\mid1-x)T(-1\mid-1)(e^{tx})=T(1\mid1-x)(e^{-tx})=\sum_{k=0}^nL_n(t)x^n
\]
where $L_n(x)$ are the Laguerre polynomials. Note that $T(1\mid
1-x)$ is the Pascal triangle:
\[
T(1\mid1-x)(e^{-tx})=\left(%
\begin{array}{cccccc}
  1 & 0 & 0 & 0 & 0 & \cdots\\
  1 & 1 & 0 & 0 & 0 & \cdots\\
  1 & 2 & 1 & 0 & 0 & \cdots\\
  1 & 3 & 3 & 1 & 0 & \cdots\\
  1 & 4 & 6 & 4 & 1 & \cdots\\
  \vdots & \vdots & \vdots & \vdots & \vdots & \ddots\\
\end{array}%
\right) \left(
 \begin{array}{c}
 1\\
 -t\\
 \frac{t^2}{2}\\
 -\frac{t^3}{6}\\
 \frac{t^4}{24}\\
 \vdots
  \end{array}
 \right)=
 \left(
 \begin{array}{c}
 1\\
 1-t\\
1-2t+ \frac{1}{2}t^2\\
 1-3t+\frac{3}{2}t^2-\frac{1}{6}t^3\\
 1-4t+3t^2-\frac{2}{3}t^3+\frac{1}{24}t^4\\
 \vdots
  \end{array}
 \right)
\]
From the definition of the polynomials we obtain easily the
well-known general term:
\[
L_n(x)=p_n(x)\star
e^{-x}=\sum_{k=0}^n\binom{n}{k}x^k\star\sum_{k\geq0}\frac{(-1)^k}{k!}x^k=
\sum_{k=0}^{n}(-1)^k\frac{1}{k!}\binom{n}{k}x^k
\]

Our recurrence relation for Laguerre polynomials is:
\[
L_n(x)=xL_{n-1}(x)\star(-\log(1-x))+L_{n-1}(x)
\]

and the recurrence relations for the coefficients are

\noindent If $k\geq1$
\[
L_{n,k}=L_{n-1,k}-\frac{1}{k}L_{n-1,k-1}
\]
If $k=0$
\[
L_{n,0}=L_{n-1,0},\qquad L_{0,0}=1
\]

Using Corollary \ref{C:DerPn} we have:
\[
L'_n(x)=L'_{n-1}(x)-L_{n-1}(x)
\]
And consequently
\[
L'_n(x)=-\sum_{k=0}^{n-1}L_k(x)
\]

\end{ejem}

\begin{ejem}
\textbf{The Hermite polynomials.} We consider
\[
\sum_{n\geq0}H_n(t)x^n=T\left(\frac{1}{2e^{x^2}} \Big|
\frac{1}{2}\right)(e^{tx})=T\left(\frac{1}{e^{x^2}} \Big|
1\right)T\left(\frac{1}{2} \Big|
\frac{1}{2}\right)(e^{tx})=T\left(\frac{1}{e^{x^2}} \Big|
1\right)(e^{2tx})=e^{2tx-x^2}
\]
\[
\left(%
\begin{array}{cccccc}
  1 & 0 & 0 & 0 & 0 & \cdots\\
  0 & 1 & 0 & 0 & 0 & \cdots\\
  -1 & 0 & 1 & 0 & 0 & \cdots\\
  0 & -1 & 0 & 1 & 0 & \cdots\\
  \frac{1}{2} & 0 & -1 & 0 & 1 & \cdots\\
  \vdots & \vdots & \vdots & \vdots & \vdots & \ddots\\
\end{array}%
\right) \left(
 \begin{array}{c}
 1\\
 2t\\
 2t^2\\
 \frac{4t^3}{3}\\
 \frac{2t^4}{3}\\
 \vdots
  \end{array}
 \right)=
 \left(
 \begin{array}{c}
 1\\
 2t\\
 2t^2-1\\
 \frac{4}{3}t^3-2t\\
 \frac{2}{3}t^4-2t^2+\frac{1}{2}\\
 \vdots
  \end{array}
 \right)
\]
If $\widetilde{H}_n(x)=n!H_n(x)$, we obtain $\widetilde{H}_n(x)$
are the usual Hermite polynomials:
\[
\begin{array}{c}
  \widetilde{H}_0(x)=1 \\
  \widetilde{H}_1(x)=2x \\
  \widetilde{H}_2(x)=4x^2-2 \\
  \widetilde{H}_3(x)=8x^3-12x \\
  \widetilde{H}_4(x)=16x^4-48x^2+12 \\
\end{array}
\]
Since \[\sum_{n\geq0}H_n(t)x^n=T\left(\frac{1}{e^{x^2}}\Big|
1\right)(e^{2tx})\]

The recurrence for the $(H_n(x))$ is:
\[
H_n(x)=xH_{n-1}(x)\star \widehat{h}(x)+f_n
\]
where
\[
\widehat{h}(x)=\sum_{n\geq1}\frac{2}{n}x^n=-2\log(1-x)\qquad \text{and}\qquad f_n=\left\{%
\begin{array}{ll}
    0, & \hbox{if $n$ is odd;} \\
    \frac{(-1)^{\frac{n}{2}}}{\left(\frac{n}{2}\right)!}, & \hbox{if $n$ is even.} \\
\end{array}%
\right.
\]

and the recurrence relations for the coefficients are

\noindent If $k\geq1$
\[
H_{n,k}=\frac{2}{k}H_{n-1,k-1}
\]
If $k=0$
\[
H_{n,0}=f_n,\qquad H_{0,0}=1
\]

Using Corollary \ref{C:DerPn} we obtain
\[
H'_n(x)=2H_{n-1}(x)
\]
or equivalently, the known relation for the $\widetilde{H}_n(x)$,
\[
\widetilde{H}'_n(x)=2n\widetilde{H}_{n-1}(x)
\]

We can also  obtain the general term for the Hermite polynomials:

\[
H_{2m}(x)=\sum_{j=0}^m\frac{(-1)^{m-j}2^{2j}}{(m-j)!(2j)!}x^{2j}
\]
\[
H_{2m+1}(x)=\sum_{j=0}^m\frac{(-1)^{m-j}2^{2j+1}}{(m-j)!(2j+1)!}x^{2j+1}
\]
 From here the known equality $\widetilde{H}_n(-x)=(-1)^n\widetilde{H}_n(x)$
 is obvious.
\end{ejem}


\
\

\

Now we are going to translate the operations in the Riordan group
to the set of Hadamard $h$-weighted families of polynomials.

Suppose that $(p_n(x))$ is the associated sequences of polynomials
to the element of the Riordan group $T(f\mid g)$
If  $\displaystyle{p_n(x)=\sum_{k=0}^np_{n,k}x^k}$, $T(f\mid
g)=(p_{n,k})_{n,k\in\N}$. Let
$\displaystyle{h(x)=\sum_{n\geq0}h_nx^n}$ be such that $h_n\neq0$
$\forall n\in\N$. So, $h$ admits a reciprocal for the Hadamard
product, we represent it by $\displaystyle{h^{(-1)_{\star}}}$. In
fact
$\displaystyle{h^{(-1)_{\star}}(x)=\sum_{n\geq0}\frac{1}{h_n}x^n}$.

Consider the set
\[
\mathcal{R}_h=\{(p_n^h(x))_{n\in\N}\ / \
(p_n(x))_{n\in\N}\in\mathcal{R}\}
\]

the following result is very easy to prove:
\begin{prop}
The function
\[
\begin{array}{cccc}
  H_h: & \mathcal{R} & \longrightarrow &  \mathcal{R}_h\\
   & (p_n(x))_{n\in\N} & \longmapsto & (p_n^h(x))_{n\in\N}
\end{array}
\]
is bijective if $h$ is a Hadamard unit in $\K[[x]]$. Consequently
the umbral composition $\sharp$ defined in $\mathcal{R}$ is
transformed into an operation $\sharp_h$ converting so
$(\mathcal{R}_h,\sharp_h)$ into a group and $H_h$ converts into a
group isomorphism. Moreover if $(s_n(x))_{n\in\N} , \
(t_n(x))_{n\in\N}\in\mathcal{R}_h$ with
$\displaystyle{s_n(x)=\sum_{k=0}^ns_{n,k}x^k}$,
$\displaystyle{t_n(x)=\sum_{k=0}^nt_{n,k}x^k\ \in\
\mathcal{R}_h}$,
$(r_n(x))_{n\in\N}=(s_n(x))_{n\in\N}\sharp_h(t_n(x))_{n\in\N}$
with $\displaystyle{r_n(x)=\sum_{k=0}^nr_{n,k}x^k}$ then
\[
r_{n,j}=\sum_{k=j}^n\frac{1}{h_k}s_{n,k}t_{k,j}
\]
\end{prop}
\begin{proof}
The first part is obvious, because if the function
\[
\begin{array}{cccc}
  G_{h^{(-1)_{\star}}}: & \mathcal{R}_h & \longrightarrow &  \mathcal{R}\\
   & (s_n(x))_{n\in\N} & \longmapsto & (s_n(x)\star h^{(-1)_{\star}})_{n\in\N}
\end{array}
\]
is the inverse, for the composition of $H_h$.

Now given  $(s_n(x))_{n\in\N} , \
(t_n(x))_{n\in\N}\in\mathcal{R}_h$ we define
$(s_n(x))_{n\in\N}\sharp_h(t_n(x))_{n\in\N}=(r_n(x))_{n\in\N}$
where $r_n(x)=H_h(p_n(x)\sharp q_n(x))$ where $s_n(x)=p_n^h(x)$,
$t_n(x)=q_n^h(x)$ for every $n\in\N$. If
$\displaystyle{p_n(x)=\sum_{k=0}^np_{n,k}x^k}$ and
$\displaystyle{q_n(x)=\sum_{k=0}^nq_{n,k}x^k}$ then if
$(p_n(x))\sharp (q_n(x))=(u_n(x))$ with
$\displaystyle{u_n(x)=\sum_{k=0}^nu_{n,k}x^k}$ then
$\displaystyle{u_{n,j}=\sum_{k=j}^np_{n,k}q_{k,j}}$. Consequently
$r_{n,j}=u_{n,j}h_j$ then
\[
r_{n,j}=\sum_{k=j}^n\frac{p_{n,k}h_kq_{k,j}h_j}{h_k}=\sum_{k=j}^n\frac{s_{n,k}t_{k,j}}{h_k}
\]
\end{proof}

Another important kind of polynomial sequences in the literature
are the sequences of binomial type \cite{Rota2} or the closely
related sequences, of convolution polynomials, see
\cite{KnuthConvo}. In fact $(s_n(x))_{n\in\N}$ is a convolution
polynomial if and only if $(n!s_n(x))_{n\in\N}$ is a sequence of
binomial type.

As one can deduce from \cite{KnuthConvo} a polynomial sequence
$(s_n(x))_{n\in\N}$ forms a convolution family if and only if
there is a formal power series
$\displaystyle{b(x)=\sum_{n\geq1}b_nx^n}$ with $b_1\neq0$ such
that $\displaystyle{e^{tb(x)}=\sum_{n\geq0}s_n(t)x^n}$. So the
convolution condition
\[
s_n(t+r)=\sum_{k=0}^ns_{n-k}(t)s_k(r)
\]
come directly from the fact that
\[
e^{tb(x)}e^{rb(x)}=e^{(t+r)b(x)}
\]
So, symbolically, the Cauchy product
\[
\left(\sum_{n\geq0}s_n(t)x^n\right)\left(\sum_{n\geq0}s_n(r)x^n\right)=\sum_{n\geq0}s_n(t+r)x^n
\]
is just the convolution condition.

Now suppose again a power series
$\displaystyle{g=\sum_{n\geq0}g_nx^n}$ with $g_0\neq0$. Then
\[
T(g\mid g)(e^{tx})=\sum_{n\geq0}s_n(t)x^n=e^{\frac{tx}{g}}
\]
Consequently we have:
\begin{thm}
A polynomial sequence $(s_n(x))_{n\in\N}$  is a convolution
sequence if and only if there is a sequence $(g_n)_{n\in\N}$ in
$\K$ with $g_0\neq0$ such that
\[
s_n(x)=\frac{1}{g_0}(xs_{n-1}(x)\star(-\log(1-x)))-\frac{g_1}{g_0}s_{n-1}(x)-\cdots
-\frac{g_{n-1}}{g_0}s_{1}(x)\ \quad \ \text{for } \ n\geq2
\]
and $s_0(x)=1$, $\displaystyle{s_1(x)=\frac{x}{g_0}}$.
\end{thm}
\begin{proof}
With the comments above it is easily proved that a polynomial
sequence $(s_n(x))_{n\in\N}$ is a convolution family if and only
if there is a series $\sum_{n\geq0}g_nx^n$ with $g_0\neq0$ such
that
\[
T(g\mid g)(e^{tx})=\sum_{n\geq0}s_n(t)x^n
\]
So $(s_n(x))_{n\in\N}$ is the $e^x$-Hadamard weighted sequence
generated by the Riordan sequence $(q_n(x))_{n\in\N}$ associated,
as in Theorem \ref{T:recu}, to the Riordan array $T(g\mid g)$.
Consequently $\displaystyle{q_0(x)=\frac{g_0}{g_0}=1}$
\[
q_n(x)=\left(\frac{x-g_1}{g_0}\right)q_{n-1}(x)-\frac{g_2}{g_0}q_{n-2}(x)\cdots-\frac{g_{n-1}}{g_0}q_1(x)-\frac{g_n}{g_0}q_0(x)+\frac{g_n}{g_0}
\]
so $\displaystyle{q_1(x)=\frac{x}{g_0}}$ and
\[
q_n(x)=\left(\frac{x-g_1}{g_0}\right)q_{n-1}(x)-\frac{g_2}{g_0}q_{n-2}(x)\cdots-\frac{g_{n-1}}{g_0}q_1(x)
\ \text{for} \ n\geq2
\]
The result follows directly multiplying Hadamard by $e^x$.
\end{proof}


As we know, \cite{Rota2}, the polynomial sequences of binomial
types are closely related to the so called delta-operator, see
\cite{Rota2}. In \cite{Rog78}, \cite{Spr94}, \cite{Merlini} it was
introduced the so called $A$-sequence associated to a Riordan
array. In our notation the $A$-sequence associated to the Riordan
array $T(f\mid g)$ is just the unique power series
$\displaystyle{A=\sum_{n\geq0}a_nx^n}$ with $a_0\neq0$ such that
$\displaystyle{A(\frac{x}{g})=\frac{1}{g}}$. As a consequence the
results in \cite{2ways} we get that $A$ is the $A$-sequence of
$T(g\mid g)$ if and only if $T(A\mid A)=T^{-1}(g\mid g)$ where the
inverse operation is taking in the Riordan group. So $A$ is the
$A$-sequence of $T(g\mid g)$ if and only if $g$ is the
$A$-sequence of $T(A\mid A)$. Let us denote by $\mathcal{D}$ to
the derivative operator on polynomials. Using Theorem 1 and
Corollary 3 in \cite{Rota2} we have
\begin{thm}
Suppose that $(s_n(x))_{n\in\N}$  is the convolution sequences
associated to the Riordan array $T(g\mid g)$. Consider the
corresponding sequence $(r_n(x))_{n\in\N}$ of binomial type, i.e.
$r_n(x)=n!s_n(x)$. Then the delta-operator $Q$ having
$(r_n(x))_{n\in\N}$  as its basic sequences is just
$\displaystyle{\frac{x}{A(x)}(\mathcal{D})}$ where $A$ is the
$A$-sequence of $T(g\mid g)$. On the opposite, if we have the
delta-operator $\displaystyle{\frac{x}{g(x)}(\mathcal{D})}$ and
$(r_n(x))_{n\in\N}$ is the basis sequence then
$(\frac{r_n(x)}{n!})_{n\in\N}$ is the convolution sequence
associated to the Riordan array $T(A \mid A)$ where $A$ is the
$A$-sequence of $T(g\mid g)$.
\end{thm}

We would like to say that in \cite{2ways} it is described a
recurrence process, related to Banach Fixed Point Theorem and to
the Lagrange inversion formula, to get
$\displaystyle{\frac{x}{A}}$ using only the series $g$.

%
%
%
%
%

Now we are going to give a characterization of a generalized
Appell sequence using linear transformations in the $\K$-linear
space $\K[x]$.

Usually a Riordan matrix is defined by means of the natural linear
action on $\K[[x]]$, in fact, a matrix $A=(a_{n,k})$ is a Riordan
matrix $T(f\mid g)$ if and only if the action of $A$ on any power
series $\alpha$ is given by $\displaystyle{T(f\mid
g)=\frac{f}{g}\alpha\left(\frac{x}{g}\right)}$. In these terms we
have
\begin{prop}
A matrix $s=(s_{n,k})$ has as associated sequence of polynomials a
generalized Appell sequence if and only if there are three power
series $\displaystyle{f=\sum_{n\geq0}f_nx^n}$,
$\displaystyle{g=\sum_{n\geq0}g_nx^n}$,
$\displaystyle{h=\sum_{n\geq0}h_nx^n}$, with $f_0,g_0\neq0$ and
$h_n\neq0, \ \forall \ n\in\N$ such that the natural linear action
induced by $s$ is given by
$s(\alpha)=\displaystyle{\frac{f(x)}{g(x)}(h\star\alpha)\left(\frac{x}{g}\right)}$
for any $\alpha\in\K[[x]]$.
\end{prop}

\begin{rmk}
From the above proposition we could develop the exponential
Riordan arrays or more generally the generalized Riordan matrices,
see \cite{WangWang}.
\end{rmk}

 {\bf Acknowledgment:} The first author was partially supported by
DGES grant MICINN-FIS2008-04921-C02-02. The second author was
partially supported by DGES grant MTM-2006-0825.

\end{document}